\RequirePackage{fix-cm}
\documentclass{svjour3}                     % onecolumn (standard format)
\smartqed  % flush right qed marks, e.g. at end of proof
\usepackage{graphicx}
\usepackage{amsmath,amssymb,amsfonts}
\usepackage{subfig}
\usepackage{algorithm}
\usepackage{algorithmic}
\usepackage[normalem]{ulem}
\usepackage{comment}
\usepackage{float}
\usepackage{color}
\usepackage{blindtext}
\usepackage{cite}
\usepackage{url}
\usepackage{color}
\usepackage{xcolor}
\usepackage{multirow}

\DeclareMathOperator*{\argmin}{\mathrm{argmin}}

\graphicspath{{./Images/}}
\begin{document}

\title{Fast Optimization Algorithm on Riemannian Manifolds and Its Application in Low-Rank Representation%\thanks{Grants or other notes
}


\author{Haoran Chen         \and
        Yanfeng Sun \and %etc.
        Junbin~Gao  \and
        Yongli~Hu
}


\institute{Haoran Chen, Yanfeng Sun, Yongli Hu. \at
              Beijing Key Laboratory of Multimedia and Intelligent Technology, College of Metropolitan Transportation, Beijing University of Technology, Beijing 100124, P. R. China \\
              Tel.: +86-130-5118-1821\\
              \email{hr\_chen@emails.bjut.edu.cn, yfsun@bjut.edu.cn, huyongli@bjut.edu.cn}           %  \\
           \and
           Junbin Gao \at
              School of Computing and Mathematics, Charles Sturt University, Bathurst, NSW 2795, Australia.\\
              \email{jbgao@csu.edu.au}
}
%\thanks{Haoran Chen, Yanfeng Sun and Yongli Hu are with Beijing Key Laboratory of Multimedia and Intelligent Technology, College of Metropolitan Transportation, Beijing University of Technology, Beijing 100124, P. R. China, e-mail: hr\_chen@emails.bjut.edu.cn,\{yfsun, huyongli\}@bjut.edu.cn}% <-this % stops a space
%\thanks{Junbin Gao is with School of Computing and Mathematics, Charles Sturt University, Bathurst, NSW 2795, Australia. e-mail:jbgao@csu.edu.au}
\date{Received: date / Accepted: date}
\maketitle
\begin{abstract}
The paper addresses the problem of optimizing a class of composite functions on Riemannian manifolds and a new first order optimization  algorithm (FOA) with a fast convergence rate is proposed. Through the theoretical analysis for FOA, it has been  proved that the algorithm has quadratic convergence. The experiments in the matrix completion task show that FOA has better performance than other first order optimization methods on Riemannian manifolds. A fast subspace pursuit method based on FOA is proposed to solve the low-rank representation model based on augmented Lagrange method on the low rank matrix variety. Experimental results on synthetic and real data sets are presented to demonstrate that both FOA and SP-RPRG(ALM) can achieve superior performance in terms of faster convergence and higher accuracy.
\end{abstract}

\section{Introduction}\label{intro}%\label{Sec:I}
Recent years have witnessed the growing attention in optimization on Riemannian manifolds. Since Riemannian optimization is directly based on the curved manifolds, one can eliminate those constraints such as orthogonality to obtain an unconstrained optimization problem that, by construction, will only use feasible points. This allows one to incorporate Riemannian geometry in the resulting optimization problems, thus producing far more accurate numerical results. The Riemannian optimization have been successfully applied in machine learning, computer vision and data mining tasks, including fixed low rank optimization \cite{BoumalAbsil2011, Vandereycken2013}, Riemannian dictionary learning \cite{HarandiHartleyShenLovellSanderson2014}, and computer vision  \cite{SuSrivastavaSouzaSarkar2014,WangShanChenGao2008, TuragaVeeraraghavanSrivastavaChellappa2011}, and tensor clustering \cite{SunGaoHongMishraYin}.

The constrained optimization on Euclidean space often has much larger dimension than that on manifolds implicitly defined by the constraints. An optimization algorithm on a manifold has therefore a lower complexity and better numerical properties. Methods of solving minimization problems on Riemannian manifolds have been extensively researched \cite{Udriste1994,AbsilMahonySepulchre2008}.  In fact, traditional optimization methods such as the steepest gradient method, conjugate gradient method and Newton method in Euclidean space can be easily adopted to optimization on Riemannian manifolds and a solid conceptual framework was built over the last decades  and the generic algorithm implementation is openly available, see \url{http://www.manopt.org}.

As one of fundamental optimization algorithms, the steepest descent method was first proposed for optimization on Riemmanian manifolds in \cite{Gabay1982}. The algorithm is computationally simple, however its convergence is very slow, particularly, for large scale complicated optimization problems in modern machine learning. In contrast, the Newton's method \cite{Gabay1982} and the BFGS quasi-Newton scheme (BFGS rank-2-update) \cite{RingWirth2012}  have higher convergence rate, however, in practical applications, using the full second-order Hessian information is computationally prohibitive.

In order to obtain a method offering high convergence rate and avoiding calculating the inverse Hessian matrix, Absil \emph{et al.} \cite{AbsilBakerGallivan2007} proposed the Trust-region method on Riemannian manifolds.
%\GaoC{Can you check Absil's book?  I think trust-region method needs Hessian matrix information.} \Chen{(Yes, solving the the trust-region subproblem needs  Hessian matrix information,$ \min\limits_{\eta\in T_{x_k}\mathcal M} f(x_k)+\langle \text{grad} f(x_k),\eta\rangle +\frac{1}{2}\langle H_k[\eta],\eta \rangle, s.t. \langle \eta,\eta\rangle \leq \Delta_k^2.$ However it uses $H_{x_k}$ in the computation  of $H_{x_k}\delta_j$ only.)}
For example, the Trust-region method has been applied the optimization problem on Grassmann manifold for the matrix completion problem \cite{BoumalAbsil2011}. Each iteration of the trust-region method involves solving  the Riemannian Newton equation \cite{Smith1993}, which increases the complexity of the algorithm. Huang \emph{et al.} \cite{HuangAbsilGallivan2015}  generalized symmetric rank-one trust-region method to a vector variable optimization problem on a $d$-dimensional Riemannian manifold, where an approximated Hessian matrix was generated by using the symmetric rank-one update without solving Riemannian Newton equation. This method has superlinear convergence. However, the symmetric rank-one update matrices are generated only by vectors  rather than matrices, so this algorithm is not feasible for the variable is the $m\times n$ matrix.  This is a great barrier in practical applications.

Generally, optimization methods which use second-order function information can get faster convergent speed than just using first-order information, however, this will greatly increase computational complexity at the same time. Fletcher-Reeves conjugate gradient method on Riemannian manifolds \cite{RingWirth2012} is a typical type of methods which need only the first-order function information, and its convergence is superlinear but lower than the desired 2-order speed.

For the gradient method, it is known that the sequence of function values $F(\mathbf X_k)$ converges to the optimal function value $F_*$ at a rate of convergence that is no worse than $\mathcal{O}(1/k)$, which is also called a superlinear convergence. For optimization on Euclidean spaces, researchers have sought for acceleration strategies to speed up the optimization algorithms requiring only first-order function information, \emph{e.g.}  based on linear search  \cite{Nesterov1983}. More strategies have been proposed by considering special structures of objectives, such as composite objective functions \cite{BeckTeboulle2009,TohYun2010,Nesterov2013}. In these strategies, the algorithms need only the first-order function information with a guaranteed quadratic convergence rate for $F(\mathbf X_k)$. In this paper we extend the method to Algorithm \ref{Algorithm1} for the general model \eqref{P} on Riemannina manifolds and we establish the improved complexity and convergence results, see Theorem \ref{thm 1} in Section \ref{Sec:III}.

As a practical application of the proposed fast gradient-like first-order algorithm, we further consider the learning problems on low-rank manifolds. Low-rank learning seeks the lowest rank representation among all the candidates that can contain most of the data samples information. The low-rank learning has been widely applied in  the low-rank matrix recovery (MR)/low-rank matrix completion \cite{MishraApuroopSepulchre2012,MishraMeyerBachSepulchre2013,MishraMeyerBonnabelSepulchre2014} and low-rank representation (LRR) for subspace clustering \cite{LiuLinYu2010,LinLiuSu2011, LiuLinYanSunYuMa2013}.  One of low-rank matrix completion application examples is the Netflix problem \cite{BennettElkanLiuSmythTikk2007}, in which one would like to recover a low-rank matrix from a sparse sampled matrix entries.

However, the lowest rank optimization  problems are NP hard and generally extremely hard to solve (and also hard to approximate \cite{RechtFazelParrilo2010}). Typical approaches are either to transform the lowest rank problem into minimization of nuclear norm by using effective convex relaxation or parametrize the fixed-rank matrix variables based on matrix decomposition.

As the set of rank-$r$ matrix $\mathcal{M}_r$ is a smooth Riemannian manifold under an appropriately chosen metric \cite{AbsilMahonySepulchre2008,MishraApuroopSepulchre2012,Vandereycken2013, MishraMeyerBonnabelSepulchre2014}, one can convert a rank constrained optimization problem to an unconstrained optimization problem on the low rank matrix manifold.  Vandereycken\cite{Vandereycken2013} considered this problem as an optimization problem on fixed-rank matrix Riemannian manifold. Considering different types of matrix decomposition, Mishra \emph{et al.} \cite{MishraApuroopSepulchre2012,MishraMeyerBonnabelSepulchre2014} proposed a Riemannian quotient manifold for low-rank matrix completion. Ngo \emph{et al.} \cite{NgoSaad2012} addressed this problem based on a scaled metric on the Grassmann manifold.

The aforementioned works transform/simplify a low-rank matrix completion problem into the fixed-rank optimization. However, low-rank optimization of matrix completion is not equivalent to the fixed-rank optimization. In order to exploit the fixed-rank optimization for low-rank optimization, Mishra \emph{et al.} \cite{MishraMeyerBachSepulchre2013} proposed a pursuit algorithm that alternates between fixed-rank optimization and rank-one updates to search the best rank value. Tan \emph{et al.} \cite{TanTsangWangVandereyckenPan2014}  proposed a Riemannian pursuit approach which converts low-rank problem into a series of fixed rank problems, and further confirmed that low-rank
problem can be considered as optimization whose search space is varieties of low-rank matrices $\mathcal{M}_{\leq r}$ \cite{TanShiGaoHengelXuXiao2015}, see \eqref{variety}, using the subspace pursuit approach on Riemannian manifold to look for the desired rank value $r$. Since $\mathcal{M}_{\leq r}$ is a closure of the Riemannian submanifold $\mathcal{M}_r$, see \eqref{lrrm}, an iteration algorithm over $\mathcal{M}_{\leq r}$ is  more safe-guarded than on the fixed-rank open manifold $\mathcal{M}_r$. We abuse "Riemannian" for simplicity. In this paper, we consider low-rank learning as optimization problems on varieties of low-rank matrix Riemannian manifolds $\mathcal{M}_{\leq r}$.

 The contributions of this paper are:
 \begin{enumerate}
   \item We develop a fast optimization algorithm (FOA) for the model \eqref{P} on general Riemannian manifolds. This algorithm only uses first-order function information, but has the convergence rate $\mathcal{O}(1/k^2)$.
   \item We provide a theoretical analysis for the proposed FOA and prove its quadratic convergence.
   \item We propose fast subspace pursuit approach (SP-RPRG(ALM)) on low rank matrix varieties based on the augmented Lagrangian method for the model \eqref{LRR}.
\end{enumerate}
The rest of this paper is organized as follows. Section \ref{Sec:II} introduces the necessary symbol and concepts. In section \ref{Sec:III}, we propose the fast optimization algorithm on Riemannian manifold, and discuss its convergence in detail. In section \ref{Sec:IV}, we extend the previous algorithm to the low-rank representation with noise on low-rank matrices varieties. Moreover, we proposed the fast subspace pursuit method (SP-RPRG(ALM)) based on augmented Lagrange approach and FOA on low rank matrices varieties and applied SP-RPRG(ALM) to this model. Some experimental results are presented for evaluating the performance of our algorithm in Section \ref{Sec:V}. Finally, the conclusions are summarized in Section \ref{Sec:VI}.
\section{Notations and Preliminaries}\label{Sec:II}
This section will briefly describe some notations and concepts that will be used throughout this paper.

\subsection{Notations}\label{Notation}
We denote matrices by boldface capital letters, \emph{e.g.}, $\mathbf{A}$, vectors by boldface lowercase letters, \emph{e.g.}, $\mathbf{a}$, and scalars by letters, \emph{e.g.}, $a$. The $\ell_{p}$-norm of a vector $\mathbf{v}$ is denoted by $\|\mathbf{v}\|_p$.  $\|\mathbf{v}\|_0$ is defined as the number of non-zero elements in the vector $\textbf{v}$. $\|\mathbf A\|_{21}$ denotes the sum of all  $\ell_2$-norm of its columns. The operator $\max\{\textbf{u},\textbf{v}\}$ means the vector resulting from the component-wise maximum of $\textbf{u}$ and $\textbf{v}$. The inner product of two matrices $\textbf{A}$ and $\textbf{B}$ at the same size is defined as $\langle \textbf{A},\textbf{B} \rangle = \text{tr}(\textbf{A}^{\textrm{T}}\textbf{B})$. The superscript \textrm{T} denotes the transpose of a vector/matrix.  The $\text{diag}(\boldsymbol{\sigma})$ is the diagonal matrix whose diagonal elements are the components of vector $\boldsymbol{\sigma}$. The SVD decomposition of $\textbf{X} \in \mathbb{R}^{m\times n}$ is denoted as $\textbf{X} = \textbf{U}\text{diag}(\boldsymbol{\sigma})\textbf{V}^{\textrm{T}}$,  where $\sigma_i$ is the $i$-th singular value of \textbf{X}. The nuclear norm of $\textbf{X}$ is defined as $\|\textbf{X}\|_* = \sum_i\sigma_i$. For any convex function $F(\textbf{X})$, its subdifferential at $\textbf{X}$ is denoted by $\partial F(\textbf{X})$.

\subsection{Riemannian Manifolds}\label{RM}
A manifold $\mathcal{M}$ of dimension $m$ \cite{Lee2003} is a topological space that locally resembles a Euclidean space $\mathbb{R}^m$  in a neighbourhood of each point $\mathbf X \in \mathcal{M}$. For example, lines and circles are $1D$ manifolds, and surfaces, such as a plane, a sphere, and a torus, are $2D$ manifolds. Geometrically, a tangent vector is a vector that is tangent to a manifold at a given point $\mathbf X$. Abstractly, a tangent vector is a functional, defined on the set of all the smooth functions over the manifold $\mathcal{M}$, which satisfies the Leibniz differential rule. All the possible tangent vectors at $\mathbf X$ constitute a Euclidean space, named the \textit{tangent space} of $\mathcal{M}$ at $\mathbf X$ and denoted by $T_{\mathbf X}\mathcal{M}$. If we have a smoothly defined metric (inner-product) across all the tangent spaces $\langle\cdot,\cdot\rangle_{\mathbf X}: T_{\mathbf X}\mathcal{M}\times T_{\mathbf X}\mathcal{M} \rightarrow \mathbb{R}$  on every point $\mathbf X \in \mathcal{M}$, then we call $\mathcal{M}$ a \textit{Riemannian manifold}.

The tangent bundle is defined as the disjoint union of all tangent spaces
\begin{equation}
T\mathcal{M} := \bigcup_{\textbf{X}\in\mathcal{M}} \{ \textbf{X}\}\times T_{\textbf{X}}\mathcal{M}.
\end{equation}
%\Chen{May The domain of the Retraction operator be tangent bundle. So }
For a Riemannian manifold, the Riemannian gradient of a smooth function $f:\mathcal{M} \rightarrow\mathbb{R}$ at $\textbf{X}\in\mathcal{M}$ is defined as the unique tangent vector $\text{grad}f(\textbf{X})$ in $T_\textbf{X} \mathcal{M}$ such that $\langle \text{grad}f(\textbf{X}),\boldsymbol{\xi}\rangle_{\mathbf X} = \text{D}f(\text{\textbf{X}})[\boldsymbol{\xi}]$ for all $\boldsymbol{\xi} \in T_{\textbf{X}}\mathcal{M}$,
where $\text{D}f(\text{\textbf{X}})$ denotes the (Euclidean) directional derivative. If $\mathcal{M}$ is embedded in a Euclidean space, then $\text{grad}f(\textbf{X})$ is given by the projection of $\text{D}f(\mathbf{X})$ onto the tangent space $T_{\mathbf X}\mathcal{M}$.

With a globally defined differential structure, manifold $\mathcal{M}$ becomes a differentiable manifold.  A \textit{geodesic} $\gamma :[0,1]\rightarrow \mathcal{M}$ is a smooth curve with a vanishing covariant derivative of its tangent vector field, and in particular, the Riemmannian distance     between two points $\mathbf X_i, \mathbf X_j \in \mathcal{M}$ is  the shortest smooth path connecting them on the manifold, that is the infimum of the lengths of all paths joining $\mathbf X_i$ and $\mathbf X_j$.

There are predominantly two operations for computations on a Riemannian manifold, namely (1) the exponential map at point $\mathbf X$, denoted by $\exp_{\mathbf X}: T_{\mathbf X}\mathcal{M} \rightarrow \mathcal{M}$,  and (2) the logarithmic map, at point $\mathbf X$, $\log_{\mathbf X}: \mathcal{M}\rightarrow T_{\mathbf X}\mathcal{M} $. The former projects a tangent vector in the tangent space onto the manifold, the latter does the reverse. Locally both mappings are diffeomorphic. Note that these maps depend on the manifold point $\mathbf X$ at which the tangent spaces are computed.

Given two points $\mathbf X_i, \mathbf X_j \in \mathcal{M}$ that are close to each other,  the distance between them can be calculated through the following formula as the norm in tangent space.
\begin{equation}
\text{dist}_{\mathcal{M}}(\mathbf X_i, \mathbf X_j)=\|\log_{\mathbf X_i}(\mathbf X_j)\|_{\mathbf X_i}
\end{equation}
The squared distance function $\text{dist}^2_{\mathcal{M}}(\mathbf X,\cdot)$ is a  smooth function for all $\mathbf X \in \mathcal{M}$. In this setting, it is tempting to define a retraction and lifting along the following lines. Given $\mathbf X \in \mathcal M$ and $\boldsymbol{\xi}\in T_{\mathbf{X}}\mathcal M$, compute \emph{Retraction operator} $R_{\mathbf X}(\boldsymbol{\xi})$ by (1) moving along $\boldsymbol{\xi}$ to get the point $\mathbf X+\boldsymbol{\xi}$ in the linear embedding space, and (2) ¡°projecting¡± the point $\mathbf X+\boldsymbol{\xi}$ back to the manifold $\mathcal{M}$. And we compute \emph{Lifting operator} $L_{\mathbf{X}_i}(\mathbf{X}_j)$ by projecting the point $\mathbf{X}_j$ onto the tangent space  $T_{\mathbf{X}_i}\mathcal{M}$. From another point of view, since $L_{\mathbf{X}_i}(\mathbf{X}_i) =\mathbf 0$, so $L_{\mathbf{X}_i}(\mathbf{X}_j)$  could be perceived as a direction denoting from  point $\mathbf X_i$ to $\mathbf X_j$ on Riemannian manifold.

At last, we need to define vector transport on $\mathcal{M}$.
\begin{equation}\label{transport}
\begin{split}
  \mathcal{T}_{\mathbf{X}\rightarrow \mathbf{Y}}: T_{\mathbf{X}}\mathcal{M} &\rightarrow T_{\mathbf{Y}}\mathcal{M},\\
 \boldsymbol{\xi} & \rightarrow P_{T_{\mathbf{Y}}\mathcal{M}}(\boldsymbol{\xi}).
\end{split}
\end{equation}

\subsection{The Riemannian Manifold of Fixed-rank Matrices}

Let
\begin{equation}\label{lrrm}
    \mathcal{M}_r = \{\mathbf{X} \in \mathbb{R}^{m\times n}: \text{rank}(\mathbf{X})=r\}.
\end{equation}
It can be proved that  $\mathcal{M}_r$ is a Riemannian manifold  of dimension $m+n-r$ \cite{Vandereycken2013}. It has the following  equivalent expression by SVD decomposition:
\begin{equation*}
\mathcal{M}_r = \{\mathbf{U}\text{diag}(\boldsymbol{\sigma}) \mathbf{V}^{\textrm{T}}:\mathbf{U}\in \text{St}_r^m, \mathbf{V}\in \text{St}_r^n, \|\boldsymbol{\sigma}\|_0 = r\}
\end{equation*}
where $\text{St}_r^m$ is the Stiefel manifold of $m\times r $  real, orthonormal matrices, and the entries in $\boldsymbol{\sigma}$ are arranged in descending order. Furthermore, the tangent space $T_{\mathbf{X}}\mathcal{M}_r$ at $\mathbf{X} = \mathbf{U}\text{diag}(\boldsymbol{\sigma})\mathbf{V}^{\textrm{T}} \in \mathcal{M}_r$ is given by
\begin{equation}
\begin{split}
 T_{\mathbf{X}}\mathcal{M}_r =  &\big\{ \mathbf{U}\mathbf{M}\mathbf{V}^{\textrm{T}}+\mathbf{U}\mathbf{V}_p^{\mathbf{T}}+\mathbf{U}_p\mathbf{V}^{\textrm{T}}:\mathbf{M}\in \mathbb{R}^{r\times r}, \\
 & \mathbf{U}_p\in \mathbb{R}^{m\times r}, \mathbf{V}_p\in \mathbb{R}^{n\times r}, \mathbf{U}_p^{\textrm{T}}\mathbf{U} =0, \mathbf{V}_p^{\textrm{T}}\mathbf{V} =0\big\}.
\end{split}
\end{equation}

 Since $\mathcal{M}_r$ is embedded in $\mathbb{R}^{m\times n}$, the Riemannian gradient of $f$ is given as the orthogonal projection of the Euclidean gradient of $f$ onto the tangent space. Here, we denote the orthogonal projection of any $\mathbf{Z} \in \mathbb{R}^{m\times n}$ onto the tangent space at $\mathbf{X} = \mathbf{U}\text{diag}(\boldsymbol{\boldsymbol{\sigma}})\mathbf{V}^{\textrm{T}} $ as
\begin{equation}\label{gradm}
\begin{split}
\textrm{P}_{T_{\mathbf{X}}\mathcal{M}_r}:\ \ &\mathbb{R}^{m\times n}\rightarrow T_{\mathbf{X}}\mathcal{M}_r\\
  & \mathbf{Z}\mapsto \textrm{P}_{\mathbf{U}}\mathbf{Z}\textrm{P}_{\mathbf{V}}+\textrm{P}_{\mathbf{U}}^{\perp}\mathbf{Z}\mathrm{P}_{\mathbf{V}}+        \textrm{P}_{\mathbf{U}}\mathbf{Z}\textrm{P}_{\mathbf{V}}^{\perp},
\end{split}
\end{equation}
where $\textrm{P}_{\mathbf{U}}:= \mathbf{U}\mathbf{U}^{\textrm{T}}, \textrm{P}_{\mathbf{U}}^{\perp}:= \mathbf{I}-\textrm{P}_{\mathbf{U}}, \textrm{P}_{\mathbf{V}}:= \mathbf{V}\mathbf{V}^{\textrm{T}}$ and $\textrm{P}_{\mathbf{V}}^{\perp}:= \mathbf{I}-\textrm{P}_{\mathbf{V}}$, for any $\mathbf{U}\in \textrm{St}_r^m,\mathbf{V}\in \textrm{St}_r^n$.

%-------------------------------------------------IID
\subsection{Varieties of Low-rank Matrices}\label{RV}
Define %The following matrix set of rank at most $r$%
\begin{equation}\label{variety}
\mathcal{M}_{\leq r} = \{\mathbf{X}\in \mathbf{R}^{m\times n}: \text{rank}(\mathbf{X}) \leq r\},
\end{equation}
then $\mathcal{M}_{\leq r}$ is the closure of the matrices set with  constant rank $\mathcal{M}_r$. At a singular point $\mathbf{X}$ where rank($\mathbf X$)$=s<r$, we have to use search directions in the tangent cone (instead of tangent space), The tangent cones of $\mathcal{M}_{\leq r}$ are explicitly known \cite{SchneiderUschmajew2015},
\begin{equation}
T_{\mathbf{X}}\mathcal{M}_{\leq r} = T_{\mathbf{X}}\mathcal{M}_s\oplus\{\Xi_{r-s}\in \mathcal{U}^{\bot}\otimes\mathcal{V}^{\bot}\},
\end{equation}
where $\mathcal{U} = \text{rang}(\mathbf{X})$ and $\mathcal{V} = \text{rang}(\mathbf{X}^{\textrm{T}})$, $\mathbf{X} = \mathbf{U}\Sigma \mathbf{V}^{\textrm{T}}$, $ \text{rank}(\mathbf{X})= s<r$.

Projection operator $\textrm{P}_{T_{\mathbf{X}}\mathcal{M}_{\leq r}}(\cdot)$, projecting an element $\boldsymbol{\eta}$ onto the tangent cone $T_{\mathbf{X}}\mathcal{M}_{\leq r}$, can be calculated as
 \begin{equation}\label{projection}
 \textrm{P}_{T_{\mathbf{X}}\mathcal{M}_{\leq r}}(\boldsymbol{\eta}) =\textrm{P}_{T_{\mathbf{X}}\mathcal{M}_{s}}(\boldsymbol{\eta})+\mathbf{\Xi}_{r-s},
\end{equation}
where
\begin{equation*}
\begin{split}
 &\textrm{P}_{T_{\mathbf{X}}\mathcal{M}_{s}}(\boldsymbol{\eta})=\textrm{P}_{\mathbf{U}}\boldsymbol{\eta}\textrm{P}_{\mathbf{V}}+\textrm{P}_{\mathbf{U}}^{\perp}
 \boldsymbol{\eta}\textrm{P}_{\mathbf{V}}
  +\textrm{P}_{\mathbf{U}}\nabla f(\mathbf{X})\textrm{P}_{\mathbf{V}}^{\perp};\\
&\mathbf{\Xi}_{r-s}=\argmin\limits_{\text{rank}(\mathbf{Y})\leq r-s} \|\boldsymbol{\eta}-\textrm{P}_{T_{\mathbf{X}}\mathcal{M}_{s}}(\boldsymbol{\eta})-\mathbf{Y} \|_F^2 .
\end{split}
\end{equation*}
Note that
$\mathbf{\Xi}_{r-s}$ is the best rank $r-s$ approximation of $\boldsymbol{\eta}-\textrm{P}_{T_{\mathbf{X}}\mathcal{M}_{s}}(\boldsymbol{\eta})$. Moreover, $\mathbf{\Xi}_{r-s}$ can be efficiently computed with the same complexity as on $\mathcal{M}_r$ in \cite{TanShiGaoHengelXuXiao2015}.

The gradient of function $f(\mathbf{X})$ can be calculated by the following  formula
\begin{equation}\label{gradv}
\text{grad}f(\mathbf{X}) = \textrm{P}_{T_{\mathbf{X}}\mathcal{M}_{\leq r}}(\nabla f(\mathbf{X})),
\end{equation}
where $\nabla f(\mathbf{X})$ denotes the Euclidean gradient of $f(\mathbf{X})$.

Another ingredient we need is the so-called retraction operator which maps a tangent vector $\mathbf{\xi}$ in the tangent cone $T_{\mathbf{X}}\mathcal{M}_{\leq r}$
at a given point $\mathbf{X}$ on $\mathcal{M}_{\leq r}$, which is defined as
\begin{equation}\label{retraction}
 \textrm{R}_{\mathbf{X}}(\boldsymbol{\xi}) = \argmin\limits_{\mathbf{Y}\in \mathcal{M}_{\leq r}}\|\mathbf{X}+\boldsymbol{\xi}-\mathbf{Y}\|_F^2.
\end{equation}

Next, we introduce the inverse of retraction operator called the Lifting operator,  denoted by $L_\mathbf{X} (\mathbf Y)$ which maps a point $\mathbf{Y}$ on manifold back to the tangent tone space $T_{\mathbf X}\mathcal M$. While the two point $\mathbf{X} , \mathbf{Y}$ are close to each other, we can compute approximately  by
\begin{equation}\label{Log}
\begin{split}
L_\mathbf{X} (\mathbf Y) : & \mathcal{M}_{\leq r}\mapsto T_{\mathbf{X}}\mathcal{M},\\
& \mathbf{Y} \rightarrow P_{T_{\mathbf{Y}}\mathcal{M}_{\leq r}}(\boldsymbol{\mathbf Y-\mathbf X}).
\end{split}
\end{equation}
Also, $L_\mathbf{X} (\mathbf Y)$ can denote a direction form $\mathbf X$ to $\mathbf Y$ on Varieties  since $L_\mathbf{X} (\mathbf X)=\mathbf 0$.
%\GaoC{Why do you distinguish $S(X,Y)$ and $L_X(Y)$? Are they same? If they are same, using two notation may cause confusion.}

%----------------------------------------------------------------III
\section{The Fast Optimization Algorithm on Riemannian manifold}\label{Sec:III}
We consider the following general model \eqref{P} on Riemannian manifold which also naturally extends the problem formulation in  \cite{BeckTeboulle2009}
\begin{equation}\label{P}
 \min\limits_{\mathbf{X}\in\mathcal{M}}F(\mathbf{X})=f(\mathbf{X})+g(\mathbf{X})
\end{equation}
where $\mathcal{M}$ is a Riemannian manifold. We make the following assumptions throughout the section:
 \begin{enumerate}
   \item $g: \mathbb{R}^{m\times n} \rightarrow \mathbb{R}$ is a continuous convex function which is possibly non-smooth.
   \item $f: \mathbb{R}^{m\times n} \rightarrow \mathbb{R}$ is a smooth convex function of the type $C^{2}$, there exists a finite real number $L(f)$ such that  $\lambda_{\max}(H)\leq L(f)$, where $\lambda_{\max}(H)$ is the largest singular value of the Hessian matrix of the function $f$.
   \item $F(\mathbf{X})$ satisfies the following inequality,
\begin{equation}\label{convexf}
  F(\mathbf{X}) \geq  F(\mathbf{Y})+\big\langle \partial F(\mathbf{X}),L_{\mathbf{Y}}(\mathbf{X}) \big\rangle
\end{equation}
where $\forall\ \  \mathbf{X},\mathbf{Y}\in \mathcal{M}$, and $L_{\mathbf{X}}(\cdot)$ is a Lifting operator defined in Section \ref{RM}.
 \end{enumerate}

It is usually difficult to directly solve \eqref{P}, in particular, in some applications such as matrix completion and low rank representation. While introducing auxiliary variables, expensive matrix inversions are often necessary. Here we adopt the following proximal Riemannian gradient (PRG) method \cite{TanShiGaoHengelXuXiao2015} to update $\mathbf{X}$.
For any $\alpha >0$, we consider the following quadratic approximation of $F(\mathbf{X}) = f(\mathbf{X})+g(\mathbf{X})$ at a given point $\mathbf{Y}\in\mathcal{M}$,
\begin{equation}\label{quadraticapproximation}
Q_{\alpha }(\mathbf{X},\mathbf{Y}) := f(\mathbf{Y}) + \langle\mathbf{grad}f(\mathbf{Y}),L_{\mathbf{Y}}(\mathbf{X})\rangle+\frac{\alpha}{2}\|L_{\mathbf{Y}}(\mathbf{X})\|_{\mathbf Y }^2+g(\mathbf{X}),
\end{equation}
Note the above local model is different from that on vector spaces \cite{TohYun2010, Nesterov2013}. We assume that the local model
admits a unique minimizer
\begin{equation}\label{oqa}
\begin{split}
P_\alpha(\mathbf{Y}): &=\argmin\limits_{\mathbf{X}\in \mathcal{M}}  Q_\alpha (\mathbf{X},\mathbf{Y})  \\
&=\argmin\limits_{\mathbf{X}\in \mathcal{M}} \big\{f(\mathbf{Y}) +\langle\mathbf{grad}f(\mathbf{Y}),L_{\mathbf{Y}}(\mathbf{X})\rangle+\frac{\alpha}{2}\|L_{\mathbf{Y}}(\mathbf{X})\|_{\textbf Y}^2+g(\mathbf{X})\big\}.
\end{split}
\end{equation}
Simple manipulation gives that
\begin{equation}\label{computerx1}
 P_\alpha (\mathbf{Y}) = \argmin\limits_{\mathbf{X}\in \mathcal{M}}\big\{g(\mathbf{X})+\frac{\alpha }{2}\|L_{\mathbf{Y}}(\mathbf{X})+\frac{1}{\alpha }\text{grad}f(\mathbf{Y})\|_{\mathbf Y}^2\big\}.
 \end{equation}
Taking $\mathbf{Y}=\mathbf{X}_{k-1}$ in \eqref{computerx1}, we define the local optimal point of \eqref{computerx1} as  $\mathbf{X}_k$,
 \begin{equation}\label{scomputerx1}
\mathbf{X}_k = P_\alpha (\mathbf{X}_{k-1}),
\end{equation}
where $1/\alpha $ plays the role of a step size. Moreover, $\alpha$ needs to satisfy the inequality
\begin{equation}\label{ineq}
F(P_{\alpha}(\mathbf{X}_{k-1}))\leq Q_{\alpha}(P_{\alpha}(\mathbf{X}_{k-1}), \mathbf{X}_{k-1}).
\end{equation}

Note that the sequence of function values $\{F(\mathbf{X}_k)\}$ produced by \eqref{scomputerx1} is nonincreasing. Indeed, for every $ k\geq 1$
\begin{equation}\label{fxk}
F(\mathbf{X}_k)\leq Q_{\alpha_k}(\mathbf{X}_k,\mathbf{X}_{k-1})\leq Q_{\alpha_k}(\mathbf{X}_{k-1},\mathbf{X}_{k-1})=F(\mathbf{X}_{k-1}).
\end{equation}

The convergent speed is  a major concern for any optimization algorithms. Linear search on Riemannian manifold  has linear convergence. This conclusion has been shown in \cite{AbsilMahonySepulchre2008} when $g(\mathbf{X}) = 0$. In order to obtain linear search with quadratic convergence, we propose a fast optimization algorithm (FOA) which extends the acceleration methods in  \cite{BeckTeboulle2009,JiYe2009, TohYun2010} onto Riemannian manifold. A specific linear combination of the previous two points $\{\mathbf{X}_{k-1},\mathbf{X}_{k-2}\}$ was employed in the acceleration methods in Euclidean space. However the linear operation is a barrier for Riemannian manifold since Riemannian manifold may be not a linear space. So we need to use necessary tools on Riemannian manifold to overcome this difficulty.

%In order to obtain a convergent sequence $\{\textbf{X}_k\}$ whose corresponding  function value sequence $\{F(\textbf{X}_k)\}$ is monotonic descent,
According to \eqref{ineq} and \eqref{fxk}, we can obtain a convergent sequence $\{\textbf{X}_k\}$ and its corresponding  function value sequence $\{F(\textbf{X}_k)\}$ which is monotonically declined. So the vector $L_{\mathbf{X}_{k-1}}(\mathbf{X}_{k-2})$ (see Section \ref{RM}) is an ascent  direction  and its negative direction is a descent direction.
We start at the point ${\mathbf{X}_{k-1}}$ and walk a specific step size along the descent direction $-L_{\mathbf{X}_{k-1}}(\mathbf{X}_{k-2})$ on $T_{\mathbf{X}_{k-1}}\mathcal{M}$ to reach a new point $\mathbf{Y}_k^*$. Then we retract $\mathbf{Y}_k^*$ onto $\mathcal M$ denoted by $R_{\mathbf{X}_{k-1}}(\mathbf Y_k^*)$, \emph{i.e.} the new  auxiliary variable $\mathbf Y_k$. This is an acceleration part of the fast optimization algorithm. Next, the point $\mathbf{X}_k$ was generated by $P_\alpha(\mathbf{Y}_k)$, \emph{i.e.}, replacing $\mathbf X_{k-1}$ with $\mathbf Y_k$  in \eqref{scomputerx1} and \eqref{ineq},
which is the difference between FOA and the linear search algorithm. In this way, we can generate the iterative sequence $\{f(\mathbf{X}_{k})\}$ with the convergence rate $\mathcal{O}(1/k^2)$.

\textbf{Stopping Conditions I}. The iteration with backtracking can be terminated if one of the following conditions is satisfied:
\begin{itemize}
  \item[] Condition 1:\ \  $(F(\mathbf{X}_{k-1})-F(\mathbf{X}_k))/F(\mathbf{X}_{k-1})\leq \epsilon_1$;
  \item[] Condition 2:\ \ $1/\alpha _k\leq \epsilon_2$;
  \item[] Condition 3:\ \  Number of iterations $\geq N_1$.
\end{itemize}
where $\epsilon_1$ and $\epsilon_2$ denote tolerance values, and $N_1$ denotes the given positive integer.

The FOA is summarized in Algorithm \ref{Algorithm1}.
\begin{algorithm}
\renewcommand{\algorithmicrequire}{\textbf{Input:}}
\renewcommand\algorithmicensure {\textbf{Output:}}
\caption{FOA on Riemannian manifold for model \eqref{P} }\label{Algorithm1}
\begin{algorithmic}[1]
\STATE   Initial Take $ \alpha_0>0 $,  $\eta >1, \beta\in(0, 1)$. Set $\mathbf{Y}_1=\mathbf{X}_0$, and $t_1 = 1$.
\STATE  Find the smallest nonegative integers $i_k$ such that with $\bar{\alpha } = \eta^{i_k}\alpha _{k-1}$
\begin{equation}\label{searchinequation}
F(P_{\bar{\alpha}}(\mathbf{Y}_k))\leq Q_{\bar{\alpha}}(P_{\bar{\alpha}}(\mathbf{Y}_k), \mathbf{Y}_k).
\end{equation}
Set $\alpha_k = \eta^{i_k}\alpha_{k-1}$ and compute
\begin{equation}\label{fastiterative1}
\begin{split}
&\mathbf{X}_k = P_{\alpha_k}(\mathbf{Y}_k),\\
&t_{k+1} = \frac{1+\sqrt{1+4t_k^2}}{2} \\
&\mathbf{Y}_{k+1} = R_{\mathbf{X}_k}\big{(}-\frac{t_k-1}{t_{k+1}}L_{\mathbf{X}_{k}}(\mathbf{X}_{k-1})\big{)}
\end{split}
\end{equation}
\STATE Terminate if stopping conditions I are achieved.
\ENSURE  $\mathbf X_k$
\end{algorithmic}
\end{algorithm}

Now we are in a position to present the major conclusion of the paper.
\begin{lemma}[\textbf{Optimality Condition}]\label{lemma0}
A point $\mathbf{X}_*\in\mathcal{M}$ is a local minimizer of \eqref{P} if and only if there exists $\boldsymbol{\eta}\in\partial g(\mathbf{X})$ such that \cite{AbsilMahonySepulchre2008}
\begin{equation*}
  \mathbf{grad}f(\mathbf{X})+\boldsymbol{\eta} = 0.
\end{equation*}
 where $ \mathbf{grad}f(\mathbf{X}),\boldsymbol{\eta} \in T_{\mathbf{X}}\mathcal{M}$.
\end{lemma}

\begin{lemma}\label{lem 1}
 Let $\mathbf{X}_k\in\mathcal{M}$ and $\boldsymbol{\eta}\in T_{\mathbf{X}_k}\mathcal{M}$ be a descent direction. Then there exists an $\alpha_k$ that satisfies the condition in \eqref{searchinequation}.
 \end{lemma}

\begin{proof}
 According to  the quality of function $F(\mathbf X)$, $\mathbf Y$ is close to $\mathbf X$, $\mathbf Y\in\mathcal{M}$, we can obtain
 $F(\mathbf X) \leq  f(\mathbf{Y}) + \langle\mathbf{grad}f(\mathbf{Y}),L_{\mathbf{Y}}(\mathbf{X})\rangle+\eta L(f)\|L_{\mathbf{Y}}(\mathbf{X})\|_{\mathbf Y }^2/2+g(\mathbf{X})$, \emph{i.e.} equality \eqref{searchinequation} to hold when $\alpha_k \geq \eta L(f)$.
 \end{proof}
%\GaoC{This lemma in [29] is for any Manifold or just for the low rank manifold?}
The following Theorem \ref{thm 1} gives the quadratic convergence of Algorithm \ref{Algorithm1}.

\begin{theorem}\label{thm 1}
Let $\{\mathbf{X}_k\} $ be generated by Algorithm~\ref{Algorithm1}, then for any $k\geq 1$
  \begin{equation*}
  F(\mathbf{X}_k)-F(\mathbf{X}_*)\leq \frac{2 \eta L(f)\|L_{\mathbf{X}_*}(\mathbf{X}_0)\|_{\mathbf{X}_*}^2}{(k+1)^2}
\end{equation*}
\end{theorem}
The proof of Theorem \ref{thm 1} is given in Appendix A.

When noise exists in data,
the model \eqref{P} can be changed into the following form:
\begin{equation}\label{PE}
  \min\limits_{\mathbf{X}\in\mathcal{M}, \mathbf{E}\in\mathbb{R}^{m\times n}}F(\mathbf{X},\mathbf{E})=f(\mathbf{X},\mathbf{E})+g(\mathbf{X},\mathbf{E}).
\end{equation}
where the  properties of function $f$, $g$ is the same as \eqref{P}.

In order to solve the problem \eqref{PE} following \cite{LinLiuSu2011}, we can optimize the variables $\mathbf{X}$ and $\mathbf{E}$ by  alternating direction  approach. This approach decomposes the minimization of $F(\mathbf{X},\mathbf{E})$ into two subproblems that minimize w.r.t. $\mathbf{X}$ and $\mathbf{E}$, respectively. The iteration of each subproblem goes as follows,
\begin{equation}\label{solveX}
  \mathbf{X}_{k} = \argmin\limits_{\mathbf{X}\in\mathcal{M}} F(\mathbf{X},\mathbf{E}_{k-1})
\end{equation}
\begin{equation}\label{solveE}
 \mathbf{E}_{k} =\argmin\limits_{\mathbf{E}\in\mathbb{R}^{m\times n}} F(\mathbf{X}_{k},\mathbf{E})
\end{equation}
While solving \eqref{solveX}, the variable $\mathbf{E}_{k-1}$ is fixed. We consider the following robust proximal Riemannian gradient (RPRG) method  of $F(\mathbf{X},\mathbf{E}_{k-1})$ at a given point $\mathbf{Y}_{k}$ on $\mathcal{M}$
\begin{equation*}
\begin{split}
Q_{\alpha}(\mathbf{X},\mathbf{Y}_{k},\mathbf{E}_{k-1}) := & f(\mathbf{Y}_{k},\mathbf{E}_{k-1})+\big{\langle}\mathbf{grad}f_{\mathbf{X}}(\mathbf{Y}_{k},\mathbf{E}_{k-1}),
L_{\mathbf{Y}_k}(\mathbf{X})\big{\rangle} \\ &+\frac{L}{2}\|L_{\mathbf{Y}_k}(\mathbf{X})\|_{\mathbf Y_k}^2 +g(\mathbf{X},\mathbf{E}_{k-1})
\end{split}
\end{equation*}
Define
\begin{align}\label{computex2}
\mathbf{X}_{k} &= P_\alpha(\mathbf{Y}_k,\mathbf E_{k-1}) := \argmin\limits_{\mathbf{X}\in \mathcal{M}} Q_\alpha(\mathbf{X},\mathbf{Y}_k,\mathbf{E}_{k-1}), \notag\\
&=\argmin\limits_{\mathbf{X}\in \mathcal{M}}\big\{g(\mathbf{X},\mathbf{E}_{k-1})
+\frac{\alpha}{2}\|L_{\mathbf{Y}_k}(\mathbf{X})+\frac{1}{\alpha}\text{grad}f_{\mathbf{X}}(\mathbf{Y}_k,\mathbf{E}_{k-1})\|\big\}
\end{align}
% and define
%  \begin{equation}
% \begin{split}
% \mathbf{X}_{k} &= P_\alpha(\mathbf{Y}_k,\mathbf E_{k-1})
% = \argmin\limits_{\mathbf{X}\in \mathcal{M}}\big\{g(\mathbf{X},\mathbf{E}_{k-1})
% \\ &+\frac{\alpha}{2}\|L_{\mathbf{Y}_k}(\mathbf{X})+\frac{1}{\alpha}\text{grad}f_{\mathbf{X}}(\mathbf{Y}_k,\mathbf{E}_{k-1})\|\big\}.
%  \end{split}
%  \end{equation}
where the point $\mathbf Y_k\in\mathcal{M}$ is generated by the previous point $(\mathbf X_{k-2}, \mathbf X_{k-1})$ in same methods as Algorithm \ref{Algorithm1}.

In most cases, subproblem \eqref{solveE} has a closed solution.

Similar to Algorithm \ref{Algorithm1}, we propose the following \textbf{stopping conditions II}:  %of Algorithm \ref{Algorithm2}. %
%The fast optimization Algorithm of model \eqref{PE} stops if one of the following conditions is achieved:
\begin{itemize}
  \item[] Condition 1: $\|\mathbf{X}_{k}-\mathbf{X}_{k+1}\| \leq \varepsilon_1$ and $ \|\mathbf{E}_{k}-\mathbf{E}_{k+1}\| \leq \varepsilon_2$;
  \item[] Condition 2:\ \  Number of iterations $\geq N_1$.
  \item[] Condition 3: The number of alternating direction iterations $\geq N_2$.
\end{itemize}
where $\varepsilon_1$ is a threshold for the change in the solutions, $\varepsilon_2$ is a threshold for the error constraint, and $N_1, N_2$ are positive integers.

Algorithm \ref{Algorithm2} summarizes the FOA for model  \eqref{PE}.
\begin{algorithm}
\renewcommand{\algorithmicrequire}{\textbf{Input:}}
\renewcommand\algorithmicensure {\textbf{Output:}}
  \caption{ FOA on Riemannian manifold for model \eqref{PE}}\label{Algorithm2}
\begin{algorithmic}[1]
\STATE  Initial Take $ \alpha_0>0 $,  $\eta >1, \beta\in(0, 1)$. $\mathbf{X}_0= \mathbf{0}, \mathbf{E}_0 = \mathbf{0}$, Set $\mathbf{Y}_1=\mathbf{X}_0$, and $t_1 = 1$.
\FOR{ $k = 1, 2,\cdots,N_2$}
\FOR{$i = 1, 2,\cdots,N_1$}
\STATE  Find the smallest nonegative integers $i_j$ such that with $\bar{\alpha} = \eta^{i_j}\alpha_{i-1}$
\begin{equation*}%\label{searchinequation2}
F\big(P_{\bar{\alpha}}(\mathbf{Y}_i), \mathbf{E}_{i-1}\big) \leq Q_{\bar{\alpha}}\big(P_{\bar{\alpha}}(\mathbf{Y}_i),\mathbf{Y}_i, \mathbf{E}_{i-1}\big)
\end{equation*}
Set $\alpha_i = \eta^{i_j}\alpha_{i-1}$ and compute
\begin{equation*}
\begin{split}
&\mathbf{X}_i = P_{\alpha_i}(\mathbf{Y}_i, \mathbf E_{i-1}),\ \  \text{by} \ \ \eqref{computex2}\\
&t_{i+1} = \frac{1+\sqrt{1+4t_i^2}}{2} \\
&\mathbf{Y}_{i+1} = R_{\mathbf{X}_i}\big(-\frac{t_i-1}{t_{i+1}}L_{\mathbf{X}_{i}}(\mathbf{X}_{i-1})\big)
\end{split}
\end{equation*}
\STATE Terminate if stopping conditions I are achieved.
\ENDFOR
\STATE $\mathbf X_k = \mathbf X_i.$
\STATE $ \mathbf{E}_k = \argmin\limits_{\mathbf{E}\in\mathbb{R}^{m\times n}}F(\mathbf{X}_{k},\mathbf{E})$
\STATE Terminate if stopping conditions II are achieved.
\ENDFOR
\ENSURE $(\mathbf X_k,\mathbf E_k).$
\end{algorithmic}
\end{algorithm}

By Theorem \ref{thm 1}, it is easy to infer the following theorem to guarantee the convergence of Algorithm \ref{Algorithm2}.
\begin{theorem}
  Let $F(\mathbf{X},\mathbf{E})=f(\mathbf{X},\mathbf{E})+g(\mathbf{X},\mathbf{E})$, and $\{(\mathbf{X}_k,\mathbf{E}_k)\}$ be an infinite sequence generated by Algorithm \ref{Algorithm2}, then it follows that $F(\mathbf X_{k+1},\mathbf E_{k+1}) \leq F(\mathbf X_{k},\mathbf E_{k})$, $\{(\mathbf{X}_k,\mathbf{E}_k)\}$ converges to a limit point $(\mathbf{X}_*, \mathbf{E}_*)$.
\end{theorem}

%--------------------------------------------------------------------IV
\section{The application of FOA in low rank representation on low rank matrix varieties}\label{Sec:IV}

Low rank representation model has been widely studied in recent years. In this section, we try to apply algorithm 2 to low rank representation model on varieties. The model is defined as follows:
%consider `1propose the efficient and accurate subspace pursuit robust proximal Riemannian gradient algorithm with augmented Lagrangian function (SP-RPRG(ALM)) to solve the original low rank representation (LRR) . SP-RPRG(ALM) reformulates the LRR model \cite{LinLiuSu2011}  as follows,
%Low rank representation has been widely studied in recent years and an elegant method is proposed  based on Riemannian manifold, i.e.,subspace pursuit approach \cite{TanTsangWangVandereyckenPan2014}. We apply the proposed fast optimization algorithm  in  subspace pursuit fashion for the low rank representation problem defined as follows \cite{LinLiuSu2011},%
\begin{equation}\label{LRR}
\begin{split}
  \min\limits_{\mathbf{X}, \mathbf{E}} &\ \ \|\mathbf{X}\|_* +\lambda\|\mathbf{E}\|_{21},\\
  s.t. &\ \  \mathbf{D}\mathbf{X}+\mathbf{E} =\mathbf{D},\ \ \mathbf{X}\in\mathcal{M}_{\leq r}.
\end{split}
\end{equation}
where $\|\mathbf X\|_*$ denotes the nuclear-norm of matrix $\mathbf{X} \in\mathbb{R}^{m\times n}$, $\lambda$ is a regularization parameter,  $\mathbf{D}$ is the data matrix, $\|\mathbf{E}\|_{21}$ is a regularizer (see Section \ref{Notation}),  and $\mathcal{M}_{\leq r}$ denotes Low rank matrix varieties of rank at most $r$.

In order to solve model \eqref{LRR}, we generalize the augmented Lagrangian methods (ALM) on Euclidean space to optimization on Low rank matrix varieties. The augmented Lagrangian function is  as follows
\begin{align} \label{ALF}
  F(\mathbf{X}, \mathbf{E}, \mathbf{U}) = &\|\mathbf{X}\|_* +\lambda\|\mathbf{E}\|_{21}+\big\langle\mathbf{U}, \mathbf{D}-\mathbf{D}
  \mathbf{X}-\mathbf{E} \big\rangle +\frac{\rho}{2}\| \mathbf{D}-\mathbf{D}
  \mathbf{X}-\mathbf{E}\|_F^2,
\end{align}
where $\mathbf{U}$ is the Lagrange multiplier, $\langle \cdot, \cdot\rangle$ is the inner product, and $\lambda$, $\rho>0$ are the penalty parameters. The augmented Lagrangian function consists of two parts:
\begin{itemize}
\item The first part is a smooth and differentiable function $f(\mathbf{X},\mathbf{E},\mathbf{U}) = \big\langle\mathbf{U}, \mathbf{D}-\mathbf{D}\mathbf{X}-\mathbf E \big\rangle+\frac{\rho}{2}\| \mathbf{D}-\mathbf{D}\mathbf{X}-\mathbf{E}\|_F^2.$
 \item The second is a continuous non-smooth functions  $g(\mathbf{X},\mathbf{E}) = \|\mathbf{X}\|_* +\lambda\|\mathbf{E}\|_{21}$.
\end{itemize}
Hence, model \eqref{LRR} is equivalent to following optimization problem
\begin{equation}\label{ALMmodel}
\min\limits_{\mathbf X\in\mathcal{M}_{\leq r}}  F(\mathbf{X}, \mathbf{E}, \mathbf{U})=\|\mathbf{X}\|_* +\lambda\|\mathbf{E}\|_{21}+\big\langle\mathbf{U}, \mathbf{D}-\mathbf{D}
  \mathbf{X}-\mathbf{E} \big\rangle +\frac{\rho}{2}\| \mathbf{D}-\mathbf{D}
  \mathbf{X}-\mathbf{E}\|_F^2,
\end{equation}
Since $\mathcal{M}_{\leq r}$ is a closure of the Riemannian submanifold  $\mathcal{M}_r$, it can be guaranteed that \eqref{ALMmodel} has optimal solution.
In spite of the non-smooth of set $\mathcal{M}_{\leq r}$, it has been shown in \cite{SchneiderUschmajew2015} that the tangent cone of $\mathcal{M}_{\leq r}$ at singular points  with rank$(\mathbf X)<r$ has a rather simple characterization and can be calculated easily (see section\ref{RV}). Hence the model \eqref{ALMmodel} can be directly optimized on $\mathcal{M}_{\leq r}$ \cite{UschmajewVandereycken2014} as section \ref{Sec:III}.
% This is a special case of model \eqref{PE}.

Subspace pursuit approach is an elegant optimization method on Riemannian manifold. In order to better optimize model \eqref{ALMmodel} , we propose a fast subspace pursuit method denoted as SP-RPRG(ALM) based on FOA.
%For many commonly cited low-rank problems, nuclear norm minimization formulation of the original rank-minimization problem do not necessarily lead to the desired solution. Degenerate solutions and multiplicity seem often or always exist. Even if a certain nuclear-norm minimization solution is a provably tight relaxation, they can possibly be meaningless in its particular context\cite{DaiLi2014}, \emph{e.g.} the problem \eqref{ALF}. The subspace pursuit approach (SP)\cite{TanTsangWangVandereyckenPan2014} based on Riemannian manifold is an elegant method.
In the fast subspace pursuit method, while  achieving the local optimum  $(\mathbf{X}_k^{r},\mathbf{E}_k)$ for the problem \eqref{ALMmodel} on $\mathcal{M}_{\leq r}$ ($r$ denotes most rank of $\mathbf X_k$, $r = 0, 1,\cdots$. $k$ is the $k$th alternating iteration, $k = 0,1,\cdots$.), we will start from the local optimum $(\mathbf{X}_k^{r},\mathbf{E}_k)$ as the initial $(\mathbf{X}_0^{r+l},\mathbf{E}_0)$ to warm-start the subsequent problem on low rank matrix varieties $\mathcal{M}_{\leq r+l}$.
%\GaoC{Please your statement, if it is on varieties, then the notation should be $\mathcal{M}_{\leq r+l}$. I am wondering here you are working on the manifold $\mathcal{M}_{r+l}$ not on $\mathcal{M}_{\leq r+l}$.}
Some related ingredients on low rank matrix verities are considered as follows (see section \ref{RV}):
\begin{itemize}
 \item Projection operator $\textrm{P}_{T_{\mathbf{X}}\mathcal{M}_{\leq r}}(\cdot)$ defined by \eqref{projection};

\item Gradient of function $f(\mathbf{X},\mathbf{E},\mathbf{U})$ on  low rank matrix verities $\mathcal{M}_{\leq r}$, \emph{i.e.}, $$\mathbf{grad}f(\mathbf{X},\mathbf{E},\mathbf{U}) = \textrm{P}_{T_{\mathbf{X}}\mathcal{M}_{\leq r}}(\nabla_{\mathbf X} f(\mathbf{X},\mathbf{E},\mathbf{U})),$$
 where $\nabla_{\mathbf X} f(\mathbf{X},\mathbf{E},\mathbf{U})) = -\mathbf D^T\mathbf{U}-\rho\mathbf D^T(\mathbf{D}-\mathbf D^T\mathbf{X}-\mathbf{E})$;

\item Retraction operator $\textrm{R}_{\mathbf{X}}(\boldsymbol{\xi})$ defined by \eqref{retraction}
 $$
 \textrm{R}_{\mathbf{X}}(\boldsymbol{\xi}) = \argmin\limits_{\mathbf{Y}\in \mathcal{M}_{\leq r}}\|\mathbf{X}+\boldsymbol{\xi}-\mathbf{Y}\|_F^2
 = \mathbf{U}\text{diag}(\boldsymbol{\sigma}_+)\mathbf{V}^T$$
where $\mathbf{X}+\boldsymbol{\xi} =\mathbf{U}\text{diag}(\boldsymbol{\sigma})\mathbf{V}^T$, $\boldsymbol{\sigma}_+$ is the first $\min \{ r,\text{rank}(\text{diag}(\boldsymbol{\sigma}))\}$ elements in the vector of $\boldsymbol{\sigma}$;
\item Lifting operator $L_{\mathbf X}(\mathbf Y)$ and the vector $L_{\mathbf X}(\mathbf Y)$ denoted by \eqref{Log}.
\end{itemize}

For the concrete function $F(\mathbf X,\mathbf E, \mathbf U)$ in \eqref{ALMmodel}, solving the problem
\begin{align*}
\mathbf X_k = \argmin\limits_{\mathbf X\in\mathcal{M}_{\leq r}}& F(\mathbf X, \mathbf E_{k-1}, \mathbf U_{k-1})\notag\\
 =\argmin\limits_{\mathbf X\in\mathcal{M}_{\leq r}} &\|\mathbf{X}\|_* + \big\langle\mathbf{U}_{k-1}, \mathbf{D}-\mathbf{D} \mathbf{X}-\mathbf{E}_{k-1}\big\rangle+\frac{\rho}{2}\| \mathbf{D}-\mathbf{D}\mathbf{X}-\mathbf{E}_{k-1}\|_F^2
\end{align*}

According to \eqref{computex2}, we can obtain
\begin{align}\label{computerx3}
  \mathbf X_k &=P_\alpha(\mathbf Y_k,\mathbf E_{k-1},\mathbf U_{k-1})= \mathbf{U}_+\text{diag}\big(\max\{\sigma_+-1/\alpha,0\}\big)\mathbf{V}_+^{\textrm{T}}.
\end{align}
where $\mathbf{U}_+\text{diag}(\sigma)_+\mathbf{V}_+^{\textrm{T}} = R_{\mathbf Y_k}(-\text{grad}f_{\mathbf{X}}(\mathbf Y_k,\mathbf E_k)/\alpha)$ (see \cite{TanShiGaoHengelXuXiao2015}), and $P_\alpha(\mathbf Y_k)$ can be efficiently computed in the sense that $R_{\mathbf Y_k}(-\text{grad}f_{\mathbf X}(\mathbf Y_k,\mathbf{E}_k)/\alpha) $ can be cheaply computed without expensive SVDs\cite{Vandereycken2013}.

Next, we directly update $\mathbf{E}$ for \eqref{ALMmodel} thought fixing the variables $\mathbf{X} = \mathbf X_k$ and $ \mathbf U= U_{k-1}$, we have
\begin{align*}
\mathbf E_k =\argmin\limits_{E\in\mathbb{R}^{m\times n}} & F(\mathbf X_k, \mathbf E, \mathbf U_{k-1})\notag\\
=\argmin\limits_{E\in\mathbb{R}^{m\times n}}& \lambda \|\mathbf{E}\|_{21} + \big\langle\mathbf{U}_{k-1}, \mathbf{D}-\mathbf{D} \mathbf{X}_{k}-\mathbf{E}\big\rangle\notag+\frac{\rho}{2}\| \mathbf{D}-\mathbf{D}\mathbf X_{k}-\mathbf E\|_F^2,
\end{align*}

let $\mathbf{W} = \mathbf{D}-\mathbf{D}\mathbf X_{k}+1/\rho \mathbf U_{k-1}$, then the closed-form solution
\begin{equation}\label{CE21}
\mathbf E_{ki} = \frac{\max\{\|\mathbf W_i\|_2-\frac{\lambda}{\rho},0\}}{\|\mathbf W_i\|_2}\mathbf W_i.
\end{equation}
where $\mathbf E_{ki}$ denotes the $i$th column of $\mathbf E_k, \forall  i$ \cite{TanShiGaoHengelXuXiao2015}.

The ALM with fast subspace pursuit on $\mathcal{M}_{\leq r}$ is described in detail in Algorithm \ref{Algorithm3}. And \textbf{Stopping conditions III}. In theory, we stop Algorithm \ref{Algorithm3} if $ r> n=\text{rank}(\mathbf{D})-1$.
\begin{algorithm}
\renewcommand{\algorithmicrequire}{\textbf{Input:}}
\renewcommand\algorithmicensure {\textbf{Output:} }
\caption{Fast subspace pursuit based on ALM ¡¡on  $\mathcal{M}_{\leq r}$ for model \eqref{LRR}} \label{Algorithm3}
\begin{algorithmic}[1]
\REQUIRE Parameters $ \lambda >0$, $\rho >0$, $0<\alpha< 1$, $\beta >1$, $l$ and $k$ is positive integers. Data matrices $\mathbf{D} \in \mathbb{R}^{m\times n}$.
\STATE   Initial $\mathbf{X}_0^0, \mathbf{E}_0^0$ and $\mathbf{U}_0^0$ are zero matrices, $r = l$, $r$ denotes the most rank of current matrix $\mathbf{X}^r$, $t_1=1$.
\FOR{$r =l:l:n$}
\FOR{$k = 1, 2,\cdots,N_2$}
\FOR{$i = 1,2,\cdots,N_1$}
 \STATE Fast optimization algorithm
 \begin{align}
 \mathbf{X}_i^r & = P_{\alpha_i}(\mathbf{Y}_i, \mathbf E_{i-1}),\notag\ \  \text{by} \ \ \eqref{computerx3}\\
  t_{i+1}& = \frac{1+\sqrt{1+4t_i^2}}{2} \notag\\
 \mathbf{Y}_{i+1} &= R_{\mathbf{X}_i}\big(-\frac{t_i-1}{t_{i+1}}S(\mathbf{X}_{i},\mathbf{X}_{i-1})\big)\notag
\end{align}
\STATE Terminate if stopping conditions I are achieved.
\STATE $\mathbf X_k^r = \mathbf X_i^r.$
\ENDFOR
\STATE Compute $\mathbf E_k$ by \eqref{CE21}.
\STATE Compute $\mathbf{U}_k = \mathbf{U}_{k-1} +\rho\big{(}\mathbf{D}-\mathbf{D}\mathbf{X}_{k}^r-\mathbf{E}_{k}\big{)}$
\STATE Compute $\rho = \min(\beta\rho,1e5)$.
\STATE Terminate if stopping conditions II are achieved.
\ENDFOR
\ENDFOR
\ENSURE  $(\mathbf{X}_k^{r},\mathbf{E}_k^{r})$.
\end{algorithmic}
\end{algorithm}

%---------------------------------------------------------------------V
\section{Experimental Results and Analysis}\label{Sec:V}
In this section, we conduct several experiment to assess the proposed algorithms. We focus on the low rank matrix completion problem on a synthetic data and clustering on some public data sets for evaluating the performance of FOA in convergent rate and error rate.
%In order to demonstrate the superiority of FOA, we compare the time cost and iteration number on matrix completion experiment on synthetic data against qGeomMC, LRGeomCG, LRGeomSD and also compare the time cost and accuracy against LRR, SP-RPRG.
All algorithms are coded in Matlab (R2014a) and implement on PC machine installed a 64\_bit operating system with an iterl(R) Core (TM) i7 CPU (3.4GHz with single-thread mode) and 4GB memory.
\begin{comment}
\GaoC{This section has been improved. Thanks!  However it is still in a sketch. I would expect you should expand it into about three pages (if this is a Computing journal submission).  For example, you cannot simply present the results in Tables, you need to give a detailed analysis what those results mean. How did you obtain those?  For those comparing methods, who program them, yourself or codes from authors or codes from internet (where)? What are the settings for you to achieve the results reported in Tables? A lot of things must be reported.    Can you find a better paper and learn how other people present their experimental results?}
\end{comment}
\subsection{Matrix Completion}
Let $\mathbf{A}\in \mathbb{R}^{m\times n}$ be a matrix whose elements are known only on a subset $\Omega$ of the complete set of indexes $\{1,\ldots m\}\times \{1,\ldots n\}$. The low-rank matrix completion problem consists of finding the matrix with lowest rank that agrees with $\mathbf{A}$ on $\Omega$. This problem can be consider as optimization with the objective function $F(\mathbf{X})=1/2\|P_{\Omega} (\mathbf{X}-\mathbf{A})\|_F^2$ , where the $P_{\Omega}$ denotes projection operator  onto ${\Omega}$, $P_{\Omega}(\mathbf{X}_{i,j}) = \mathbf{X}_{i,j}$ if $(i,j)\in\Omega$, otherwise $0$.

Following \cite{Vandereycken2013},  we generate ground-truth low-rank matrices $\mathbf{A}=\mathbf{L}\mathbf{R}\in\mathbb{R}^{m\times n}$ of rank $r$,
where $\mathbf{L}\in \mathbb{R}^{m\times r}$,$\mathbf{R}\in \mathbb{R}^{r\times n}$ is  rank r matrix generated randomly. The size of the complete set $\Omega$  is $r(m+n-r)$OS with  the oversampling factor OS $> 2$.
 \begin{figure}
   \begin{center}
{\includegraphics[width=3.0 in] {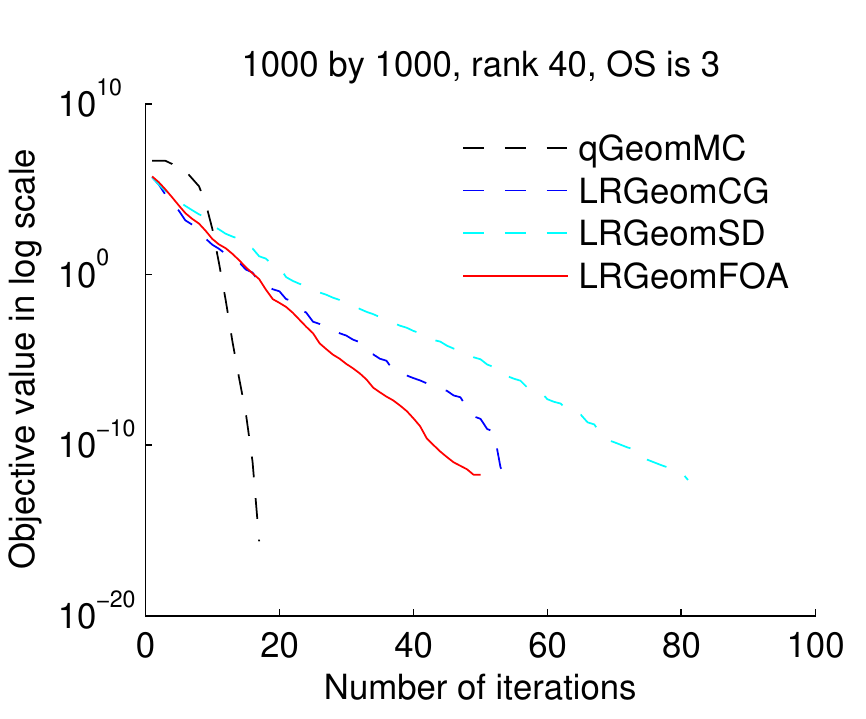}}
    \end{center}
    \caption{Objective function value w.r.t. iteration number}\label{MC1}
\end{figure}

\begin{figure}
   \begin{center}
{\includegraphics[width=3.0 in]{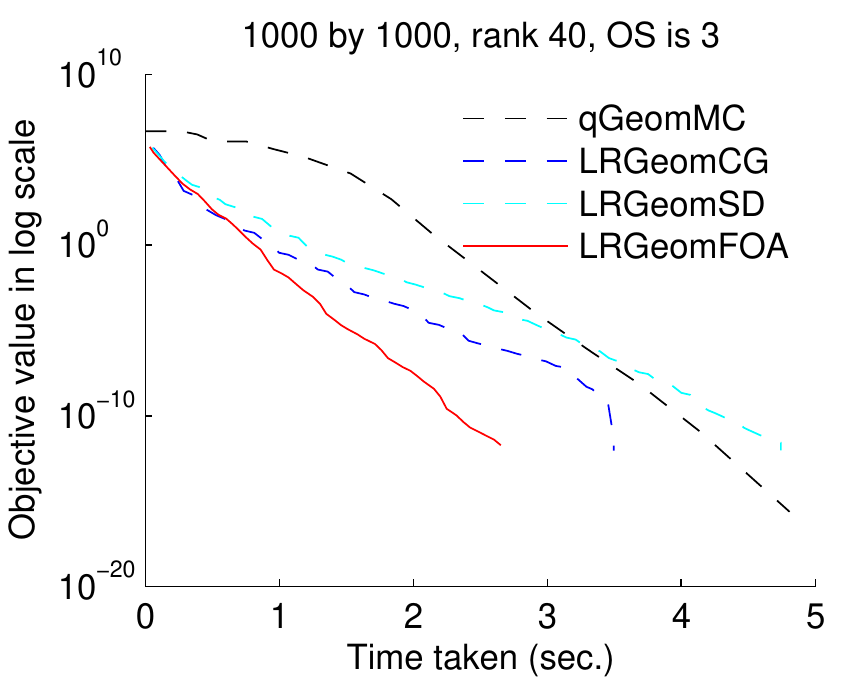}}
    \end{center}
    \caption{Objective function value w.r.t. time cost (second) }\label{MC2}
\end{figure}
\begin{figure}
   \begin{center}
{\includegraphics[width=3.0 in]{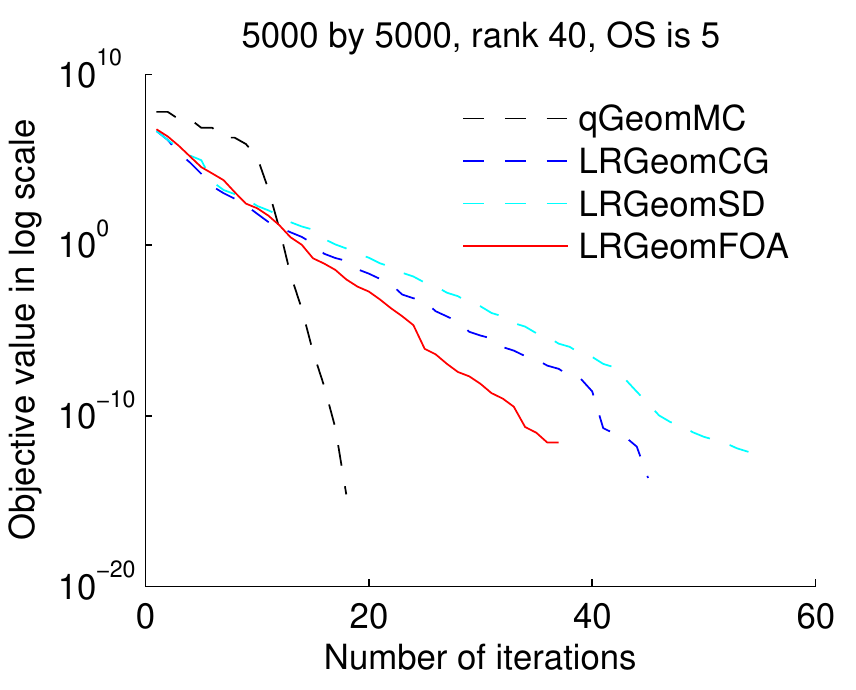}}
    \end{center}
    \caption{Objective function value w.r.t. iteration number}\label{MC3}
\end{figure}
\begin{figure}
   \begin{center}
{\includegraphics[width=3.0 in]{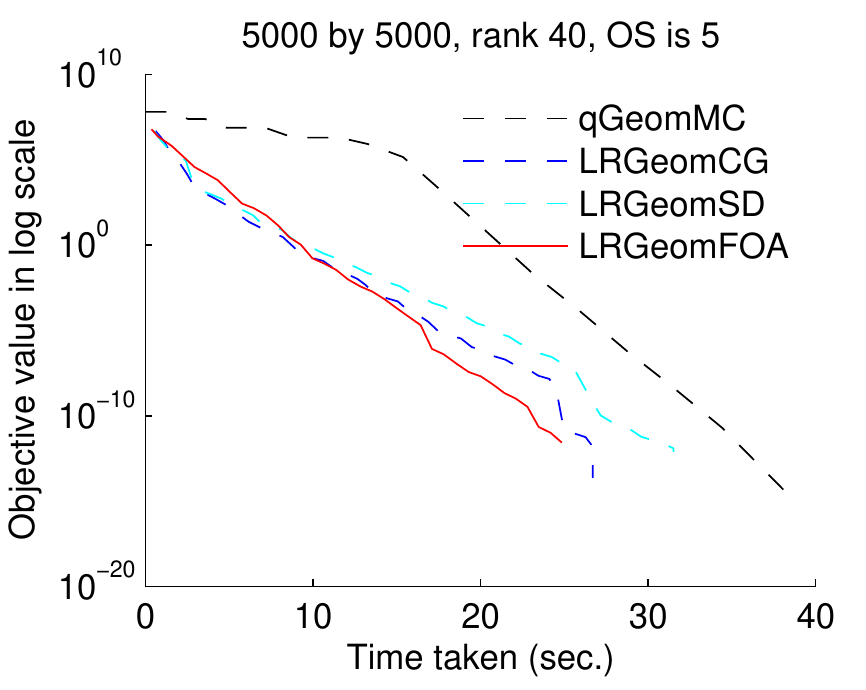}}
    \end{center}
    \caption{Objective function value w.r.t. time cost(second) }\label{MC4}
\end{figure}
 We test FOA on the matrix completion task and compare their optimization performance against qGeomMC\cite{MishraMeyerBachSepulchre2013}, LRGeomCG \cite{Vandereycken2013}, and LRGeomSD (steepest descent method) on fixed-rank Riemannian manifold.\footnote{qGeomMC and LRGeomCG are from http://bamdevmishra.com/ codes /fixedranklist/ and http://www.unige.ch/math/vandereycken/matrix  \_completion.html, respectively.} Fig.\ref{MC1} and Fig.\ref{MC3} give the needed iteration number while stop condition is satisfied with $\epsilon = 10^{-12}$ for different matrix size. Fig.\ref{MC2} and Fig.\ref{MC4} give the time cost in the same stop condition for different matrix size. We can find from these figures that qGeomMC need the least iteration number, ours is second. The convergent rate of our algorithm is the fastest. But qGeomMC employ the trust region method which uses the second-order function information. So the proposed FOA has the advantage in convergence rapidity over existing optimization algorithm with first order function information on Riemannian manifold.

\begin{figure}[H]
   \begin{center}
{\includegraphics[width=3.0 in]{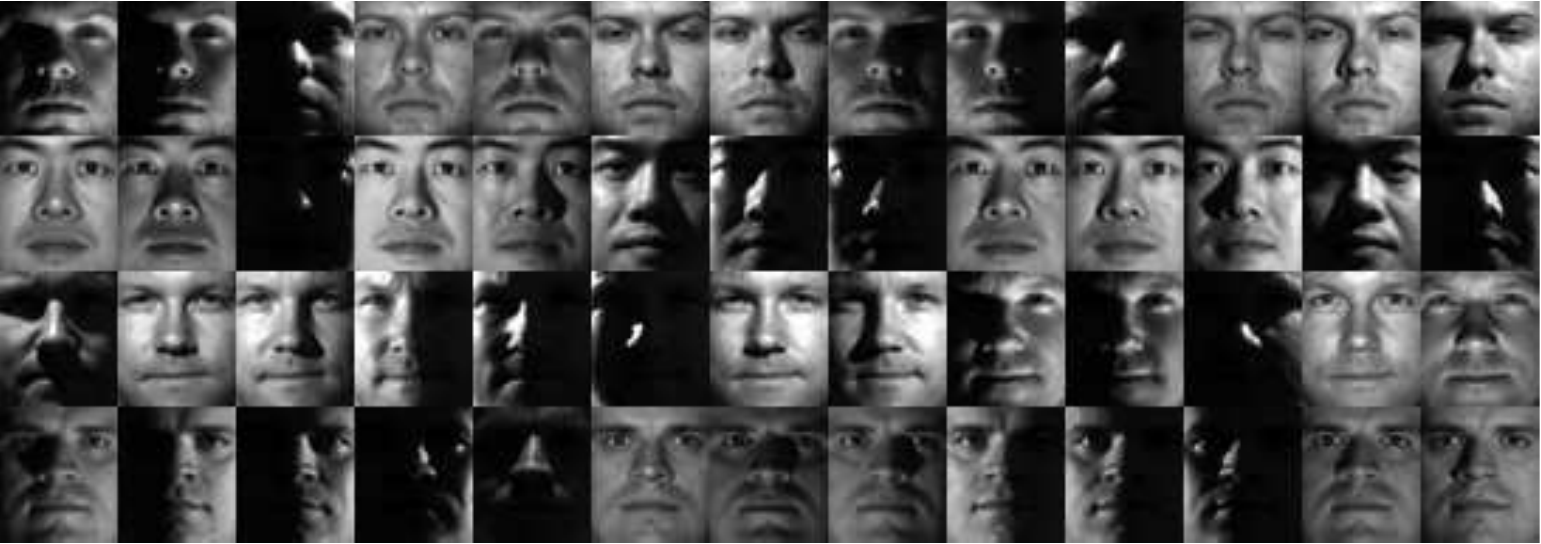}}
    \end{center}
    \caption{Some sample images from the Extended Yale B face databases}\label{Yalesam}
\end{figure}

\begin{figure}[H]
   \begin{center}
{\includegraphics[width=3.0 in]{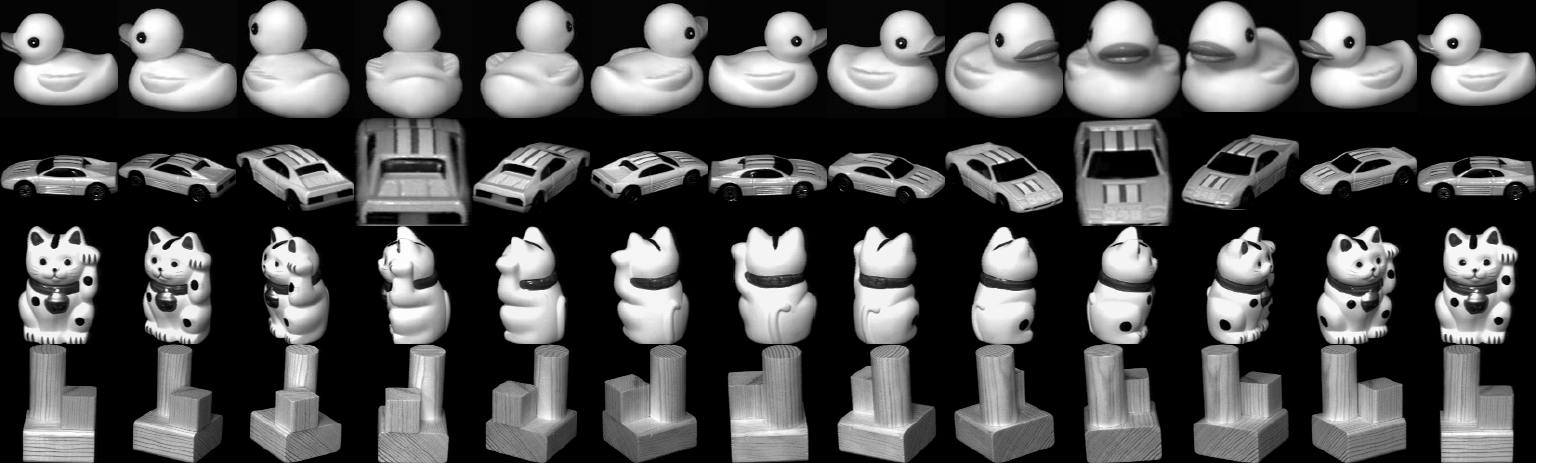}}
    \end{center}
    \caption{Some sample images from the COIL-20 databases}\label{coil}
\end{figure}

\subsection{Clustering on Extended Yale B  and COIL-20 Database}

In this subsection, we apply FOA to model \eqref{ALMmodel} to evaluate its performance in calculating time and clustering error by compared with existing optimization method on Riemannian manifold. The compared algorithm include LRR\cite{LiuLinYanSunYuMa2013}(solving model \eqref{LRR} on euclidean space), SP-RPRG\cite{TanShiGaoHengelXuXiao2015}(solving model \eqref{LRR} based on low rank matrix varieties) and SP-RPRG(ALM)(solving model \eqref{LRR} based on augmented Lagrange method on low rank matrix vatieties). For SP-RPRG and SP-RPRG(ALM), we give respectively optimal solution by conjugate gradient method and FOA method, and compare their time cost. We set $\gamma(\mathbf{E}) = \|\mathbf{E}\|_{21}$, $\mathcal{A}(\mathbf{X}) = \mathbf{D}\mathbf{X}$,  $\mathcal{B}(\mathbf{E})=\mathbf{E}$ for model \eqref{LRR}.

Two public databases are used for evaluation. % \GaoC{Evaluating which method. It is better to refer a model in the previous sections.}.
One is the extended Yale B face database which contains 2,414 frontal face images of 38 individuals. The images were captured under different poses and illumination conditions. Fig. \ref{Yalesam} shows some sample images. To reduce the computational cost and the memory requirements, all face images are resized from $192\times 168$ pixels to $48\times 42$ pixels and then vectorized as 2016-dimensional vectors.

Another database is Columbia Object Image Library (COIL-20) which includes 1,440 gray-scale images of 20 objects (72 images per object). The objects have a wide variety in geometric and reflectance characteristics. Each image is clipped out from the black background using a rectangular bounding box and resized from   $128\times 128$ pixels to $32\times 32$ pixels, then shaped as a 1024-dimensional gray-level intensity feature. Fig. \ref{coil} shows some sample images.

%We select the low rank representation model (LRR) \eqref{LRR} and model (SP-LRR) for testing the performance of our algorithm. \GaoC{If you write it in this way, nobody knows what you are talking about. In Section IV, you are talking about model (34) and propose an algorithm for it. Here you need tell readers for LRR or SP-LRR what (34) look like. Thus readers know what are the objective functions and how algorithm is applied. Always remember that the back-and-forth of the paper should be consistent. For example, after reading this paragraph, I still don't know what SP-LRR is??} For model SP-LRR, the approach based on the augmented Lagrangian method \eqref{ALF} is denoted as SP-LRR(ALM).

In the clustering experiments on extended Yale B database, we initialize parameters  $\lambda =0.001$, $\rho = 0.5$ for SP-RPRG(ALM), $\lambda = 0.1$ for LRR \footnote{LRR code is from http://www.cis.pku.edu.cn/faculty/vision/zlin/zlin.htm} \cite{LiuLinYanSunYuMa2013}, and $\lambda =0.01$, $\rho = 1$ for SP-RPRG \cite{TanShiGaoHengelXuXiao2015}. The first $c = \{2, 3, 5, 8,10\}$ classes are select for experiments, each class contains 64 images.
\begin{figure}[H]
   \begin{center}
{\includegraphics[width=3.0 in]{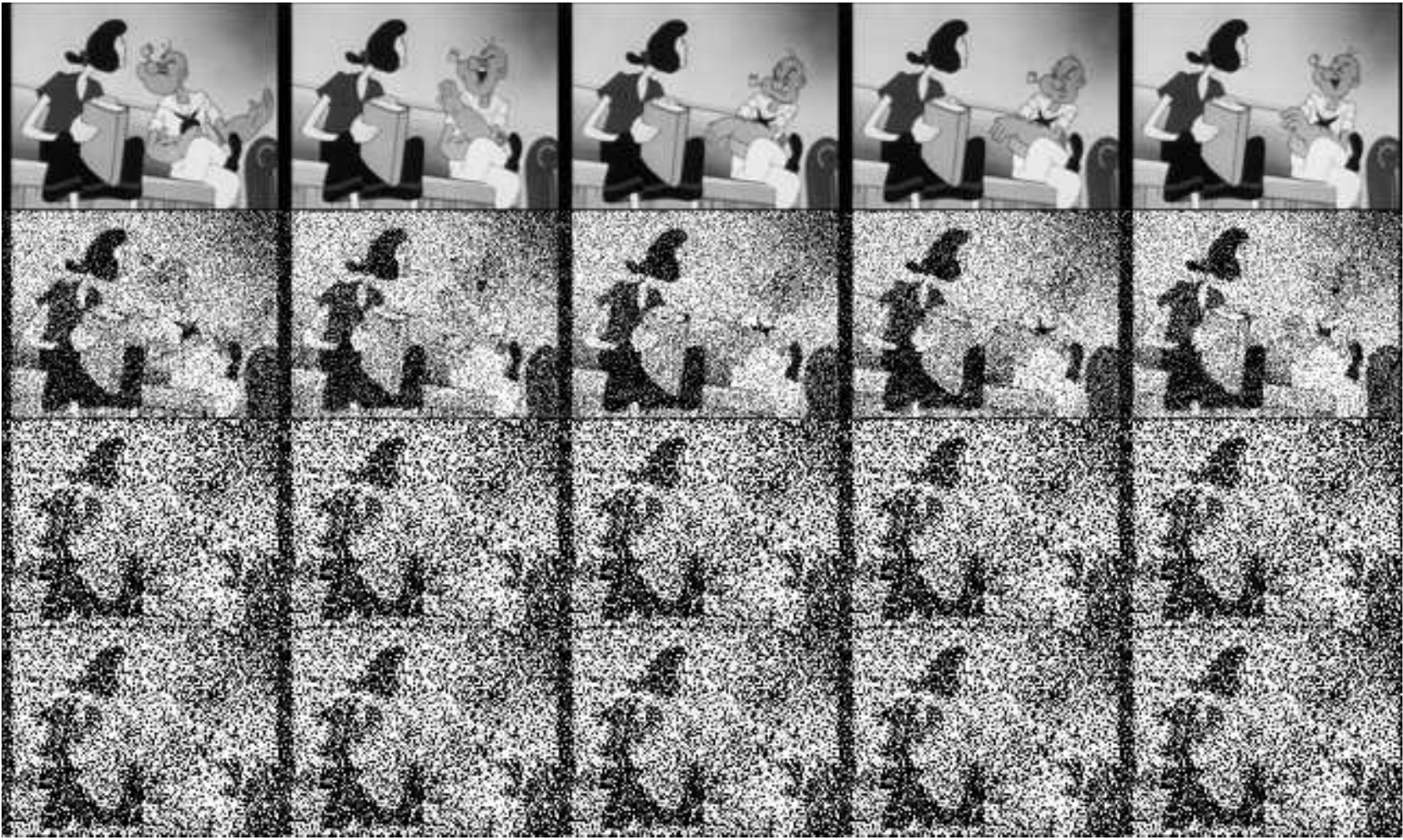}}
    \end{center}
    \caption{Some video frame examples from the Video 2  with 0, 20\%, 50\% and 80\%  Gaussian noise respectively.}\label{Video2}
\end{figure}

We compare face clustering error rate for LRR, SP-RPRG and SP-RPRG(ALM), and time cost using our FOA and conjugate gradient (CG) method for SP-RPRG and SP-RPRG  (ALM) with different $c$ value. Table \ref{Yale} gives the experimental results. We can observe from these results that the  proposed SP-LRR(ALM) has the lowest error rate for each $c$ value. FOA applied in SP-RPRG and SP-RPRG(ALM) spend less time than conjugate gradient method to achieve the same accuracy.

\begin{table*}
  \centering
   \begin{tabular}{|c|c|c|c|c|c|c|c|}
     \hline
     \multirow{2}*{Class}   &\multicolumn{1}{c|}{LRR}  &\multicolumn{3}{c|}{SP-RPRG}&\multicolumn{3}{c|}{SP-RPRG(ALM)}  \\
     \cline{2-8}
             & error& error& time(CG) & time(FOA) & error  & time(CG) & time(FOA)\\
  \hline
  \emph{2} & 2.34& 0   &71.3&14.3  &$\mathbf{0}$  & 9.0 & 3.9 \\
    \hline
        3 &3.65 & 3.13 &185.4 &56.2 &$\mathbf{1.56}$  &42.8 & 30.2\\
  \hline
          5 &4.69 & 2.50 & 581.5&187.5 &$\mathbf{2.50}$ &163.0& 160.5\\
  \hline
            8 &13.40  & 5.66 &2339.2 & 943.2&$\mathbf{3.71}$ &1027.1 &916.3 \\
   \hline
            10 &24.02  & 7.50  & 14826.9& 2064.3  &$\mathbf{3.75}$ &870 & 844.0 \\
  \hline
\end{tabular}
    \caption{Face clustering error (\%) and time(second) cost on Extended Yale B database. }\label{Yale}
\end{table*}

\begin{table*}
  \centering
   \begin{tabular}{|c|c|c|c|c|c|c|c|c|c|c|}
     \hline
      \multirow{3}*{\text{Class}}   &\multicolumn{2}{c|}{LRR}  &\multicolumn{4}{c|}{SP-RPRG}&\multicolumn{4}{c|}{SP-RPRG(ALM)}  \\

     \cline{2-11}
             & ave err  & std & ave err  & std & time & time & ave err  & std & time & time\\
             & &  &&&  (CG)& (FOA)& &   &  (CG)&(FOA)\\

    \hline
     2 & 4.08  & 11.49 &2.75    &8.11  & 0.8&0.4 &\textbf{1.94}  & 6.91&0.8 & 0.5 \\
     \hline
       3 & 6.89  & 12.81 & 4.57  & 8.00 &3.1 &1.3 & \textbf{2.65} &7.61 &3.6 &1.1  \\
     \hline
      4 & 10.92  & 13.07 & 8.46  & 9.95 &16.8 &4.0 & \textbf{4.79}&7.23 &18.6 & 7.6 \\
     \hline
   5 & 13.10   & 13.58 &9.69    &9.57  &31.9 &10.9 &\textbf{6.09}  & 6.98& 31.8& 12.3 \\
   \hline
      6 &19.67    & 13.12 & 11.98   &9.40  &35.1 & 26.2&\textbf{6.26}  & 7.03& 62.4& 39.7 \\
     \hline
       7 & 24.44   & 12.25 & 14.11   & 9.53 &63.5 & 47.3&\textbf{ 8.40}&8.19& 73.1 & 63.7 \\
     \hline
        8 & 25.02  & 11.57 & 14.87   &7.43  &91.6 &73.4 & \textbf{9.15} & 6.68&106.7  &92.1 \\
     \hline
     9 & 30.88   & 9.40 & 15.86   & 5.79 & 125.0&111.0 & \textbf{11.74} &6.15 & 132.6& 126.2 \\
     \hline
     10& 31.44   & 9.93 & 20.96    &7.43 &171.5 &165.2 &\textbf{13.57}  & 7.18&185.0 & 149.2 \\
     \hline
     11 & 33.69   & 8.56& 20.60     & 6.27& 260.1 &250.8 &\textbf{15.23}  & 6.76&270.0&245.9 \\
     \hline
   \end{tabular}
  \caption{The object clustering average error (\%), standard  deviation, and average time(second) cost and on COIL-20.}\label{Coil}
  \end{table*}

In the clustering experiments on COIL-20 database, we set the initialization parameters $\lambda =0.001$, $\rho = 1$ for SP-LRR(ALM), $\lambda = 0.1$ for LRR   and $\lambda = 0.001$, $\rho =2$ for SP-LRR. In experiment, we randomly select  $c$ (ranging from 2 to 11) classes from the 20 classes and 36 samples for each class in the whole COIL-20 databases. For each value $c$, we run 50 times using different randomly chosen classes and samples. The results are shown in Table \ref{Coil}. From this table we can see that our proposed SP-LRR(ALM) has the lowest average clustering error. FOA applied in SP-LRR and SP-LRR(ALM) spend less time than conjugate gradient method to achieve the same accuracy.This results demonstrates that the SP-LRR(ALM)  significantly enhances the accuracy and FOA reduce the time cos during the affinity matrix construction.

%SP-RPRG and our proposed method on low rank matrices varieties are better than LRR  in Euclidean space with a margin of 2 to 18 percentages. Our proposed method is better than SP-RPRG with a margin of 2 to 4 percentages.  At the same time, FOA applied in  SP-RPRG and our proposed methods have higher convergence  speed than conjugate gradient method. This results demonstrates that the SP-LRR(ALM)  significantly enhances the accuracy and  FOA reduce the time taken during the affinity matrix construction.

\subsection{Clustering on Video Scene Database}

This experiment test the performance of FOA based on model \eqref{ALMmodel} on video scene database. The compared optimization method and performance are same as previous experiments.
%In addition to compare optimization method and performance are same as previous experiments,  the robust performance also has drawn considerable attention in this experiments.

The aim of experiment is to segment individual scenes from a video sequence. The video sequences are drawn from two short animations freely available from the Internet Archive \footnote{http://archive.org/.}, this are same as the data used in \cite{TierneyGaoGuo2014}. The sequence is around 10 seconds in length (approximately 300 frames) and contains three scenes each. To segmented scenes  according to significant translation and morphing of objects  and sometimes

camera or perspective changes, as a result, there are 19 and 24 clips for Video 1 and Video 2 respectively, each clip include 3 scenes. Scene changes (or keyframes) were collected and labeled by hand to form ground truth data. The pre-processing of a sequence consists of converting colour video to grayscale and down sampling to a resolution of $129\times 96$. %\GaoC{You did not tell what model you are working. }
The video clips were corrupted with various magnitudes of Gaussian noise to evaluate the clustering performance of different methods. Fig.\ref{Video2} show some frames from video 2.

Each frame of clips is vectorized as $\mathbf X_i\in \mathbb R^{12384}$  and concatenated with consecutive frames to form the samples $\mathbf X\in \mathbb R^{12384\times 300}$ for segmentation. The model parameters are set $\lambda = 0.0002$ and $\rho = 2$ in SP-RPRG(ALM), $\lambda = 0.01$  for LRR, $\lambda = 0.01$ and $\rho =1 $ for SP-RPRG.

 %$\lambda_1 = 0.005$, $\lambda_2 = 10$ for OSC
  \begin{table*}[htbp]
  \centering
  \begin{tabular}{|c*{8}{c|}}
    \hline
\multicolumn{2}{|c}{\multirow{4}*{Gaussian noise}} &  \multicolumn{7}{|c|}{Video 1} \\
        \cline{3-9}\multicolumn{2}{|c|}{}  & LRR &  \multicolumn{3}{|c|}{SP-RPRG} & \multicolumn{3}{|c|}{SP-RPRG(ALM)}   \\
        \cline{3-9}\multicolumn{2}{|c|}{} & error &error& time&time&error&time&time\\
          \multicolumn{2}{|c|}{} & &    &(CG)& (FOA)&    &  (CG)&(FOA)\\
    \hline
\multirow{5}*{0\%}
      &min& 0.00(7)&  0.00(11)&& & 0.00(\textbf{12})&&\\% &0.00(22)&     0.00(22)& &&0.00(\textbf{22})&&\\
      &max &48.28& 29.33& && 25.59&&\\% &28.29&       16.33&  3.98\\
      &med & 22.73&0.00&&& 0.00   &&  \\%0.00&   0.00&0 .00\\
      &mean &18.82 & 9.11 &157.1 &113.8&\textbf{5.49} &49.4&21.9\\%& 1.43  &1.20 & \textbf{0.19}\\
       &std &17.63  &  13.40&& &  9.42&&\\%5.85 &4.10 & 0.82 \\
      \hline
\multirow{5}*{20\%}
      &min& 0.00(6) &0.00(11)&& & 0.00(\textbf{12}) && \\%0(14) &0(22)& 0(\textbf{22})\\
      &max &58.33 &37.93&& &31.03&&\\%&57.74  &0.59&  0.59\\
      &med & 27.16  & 0.00&&& 0.00 &&  \\%  & 0  & 0& 0\\
      &mean &20.18  & 9.38&155.8&89.9& \textbf{7.55} &50.5&18.7\\%  &12.07 &0.04 & \textbf{0.04}\\
       &std &17.36 & 13.51&&& 11.18 &&\\%  & 16.42 &0.14& 0.14\\
         \hline
 \multirow{5}*{50\%}
      &min&0.00 (6)  &   0.00(\textbf{12})&&  &0.00(18)&&\\%&1.78& 0(21)&  0(22)\\
      &max &41.98 & 25.86&& &  40.69& &\\%0.59  &0.59&46.33  \\
      &med &   13.63& 4.85 &&& 0.00    &&\\% & 0 & 0& 0\\
      &mean & 16.29 &12.77&116.2& 82.2& \textbf{6.73}  &51.9&16.4\\% &4.99 &0.06& \textbf{0.04}\\
       &std & 15.84  &15.34 &&&9.98 &&\\% & 11.80 &0.16 &0.14 \\
      \hline
 \multirow{5}*{80\%}
      &min&  0.00(5) &0.00(9)&&&  0.00(\textbf{10}) &&\\% & 0(15) &0(18)& 0(\textbf{22})\\
      &max & 52.64  & 53.41&&&26.67 &&\\%&54.83  &22.58&0.59 \\
      &med & 20.36 &2.20&&&0.00    & & \\% 0 & 0& 0\\
      &mean &   20.28 &16.39&107.6&91.7&\textbf{7.83}&56.8&18.6\\%&9.62 & 3.16&\textbf{0.004} \\
       &std &18.80  &20.00&& & 10.34  &&\\% &16.94 &7.07 &0.14 \\
      \hline
  \end{tabular}
  \caption{Average time(second) cost and clustering error rates (\%) for the corrupted video 1 with 0, 20\%, 50\% and 80\% Gaussian noise, lower is better.  Numbers in brackets indicate how many times clustering was perfect, i.e. zero error.}\label{tVideo1}
  \end{table*}

  \begin{table*}[htbp]
  \centering
  \begin{tabular}{|c*{8}{c|}}
    \hline
\multicolumn{2}{|c}{\multirow{4}*{Gaussian noise}} &  \multicolumn{7}{|c|}{Video 2} \\
        \cline{3-9}\multicolumn{2}{|c|}{}  & LRR &  \multicolumn{3}{|c|}{SP-RPRG} & \multicolumn{3}{|c|}{SP-RPRG(ALM)}   \\
        \cline{3-9}\multicolumn{2}{|c|}{} & error &error& time&time&error&time&time\\
          \multicolumn{2}{|c|}{} & &    &(CG)& (FOA)&    &  (CG)&(FOA)\\
    \hline
\multirow{5}*{0\%}
      &min &0.00(22)&  0.00(22)& &&0.00(\textbf{22})&&\\
      &max  &28.29&  16.33&  &&3.98&&\\
      &med &0.00&   0.00&&&0 .00&&\\
      &mean & 1.43  &1.20 &56.3&41.6& \textbf{0.19}&17.7&17.8\\
       &std &5.85 &4.10&& & 0.82&& \\
      \hline
\multirow{5}*{20\%}
      &min&0.00(14) &0.00(22)&&& 0.00(\textbf{22})&&\\
      &max &57.74  &0.59&&&  0.59&&\\
      &med   & 0.00  & 0.00& &&0.00&&\\
      &mean &12.07 &0.04 &51.5&39.4& \textbf{0.04}&17.6&17.3\\
       &std & 16.42 &0.14&&& 0.14&&\\
         \hline
 \multirow{5}*{50\%}
      &min&1.78& 0(21)&&&  0(22)&&\\
      &max &0.59  &0.59&&&46.33 && \\
      &med  & 0.00 & 0.00&&& 0.00&&\\
      &mean  &4.99 &0.06&45.9&37.2& \textbf{0.04}&17.7&13.3\\
       &std  & 11.80 &0.16 &&&0.14&& \\
      \hline
 \multirow{5}*{80\%}
      &min & 0.00(15) &0.00(18)& &&0.00(\textbf{22})&&\\
      &max &54.83  &22.58&&&0.59&& \\
      &med &0.00 & 0.00&&& 0.00&&\\
      &mean &9.62 &3.16&40.1&34.8&\textbf{0.004} &17.8&10.6\\
       &std  &16.94 &7.07&& &0.14&& \\
      \hline
  \end{tabular}
  \caption{Average time(second) cost and clustering error rates (\%) for the corrupted video 2 with 0, 20\%, 50\% and 80\% Gaussian noise, lower is better. Numbers in brackets indicate how many times clustering was perfect, i.e. zero error.}\label{tVideo2}

  \end{table*}

TABLE \ref{tVideo1} and \ref{tVideo2} list the segmentation results and average time cost by our FOA and conjugate gradient (CG) method for SP-RPRG and  SP-RPRG(ALM). These results are the average segmentation value of all clips for each Video.  We can see from these results that  SP-RPRG (ALM) has consistently low average error rates and standard deviation with various magnitudes of Gaussian noise. In particular, SP-RPRG (ALM)  has not been influenced by Gaussian noise in the video 2. In most case, the FOA has fast convergent rate for SP-RPRG and SP-RPRG(ALM).

%------------------------------------------------------------------VI
\section{Conclusions}\label{Sec:VI}
This paper propose a fast optimization algorithm (FOA) on a Riemannian manifold for a class of composite function, and prove its convergence. Different from most optimization methods on Riemannian manifold, our algorithm use only first-order function information, but has the convergence rate $\mathcal{O}(k^{-2})$. Experiments on some data set show the optimization performance  largely outperform existing method in terms of convergence rate and accuracy. Besides, we transform low rank representation model into optimization problem on varieties based on augmented Lagrange approach, then fast subspace pursuit methods based on FOA is applied to solve  optimization problem. % At the same time, we use subspace pursuit to look for a suitable rank of affinity matrix.%
Extensive experiments results demonstrate the superiority of our  proposed ALM with fast subspace purist approach .

\begin{acknowledgements}
The research project is supported by the Australian Research Council (ARC) through the grant DP140102270 and also partially supported by National Natural Science Foundation of China under Grant No. 61390510, 61133003, 61370119, 61171169 and 61227004.
\end{acknowledgements}


%%\bibliographystyle{spphys}       % APS-like style for physics
%\bibliography{reference_haoran}   % name your BibTeX data base


\section*{Appendix A}\label{Sec:VIII}
The requirement that the vector transport $\mathcal{T}_s$ is isometric means that, for all $\mathbf{X},\mathbf{Y}\in \mathcal{M}$ and all $\xi_{\mathbf{X}}, \zeta_{\mathbf{X}}\in \mathcal{T}_{\mathbf{X}}\mathcal{M}$, the equation $\langle\mathcal{T}_{\mathbf{X}\rightarrow\mathbf{Y}}\xi_{\mathbf{X}} ,\mathcal{T}_{\mathbf{X}\rightarrow\mathbf{Y}}\zeta_{\mathbf{X}}\rangle = \langle \xi_{\mathbf{X}}, \zeta_{\mathbf{X}}\rangle$ holds.To prove Theorem \ref{thm 1}, we give some lemmas in following.

 \begin{lemma}\label{lem 2}
  Let $a_k,b_k$ be positive real sequences satisfying
\begin{equation}\label{akbk}
  a_k- a_{k+1}\geq b_{k+1}-b_{k}, \forall k\geq 1,
\end{equation}
with $a_1 +b_1\leq c,$ $ c>0$.
Then $a_k\leq c$ for every $k\geq 1$.
\end{lemma}
 \begin{proof}
Let $c_k = a_k+ b_k$, we can obtain from \eqref{akbk}
$$a_k+b_k\geq a_{k+1}+b_{k+1}$$
i.e. $c_k\geq c_{k+1}$, so $c_k$ is a decreasing sequence. Then
$$c_k = a_k+b_k \leq c_1= a_1+b_1 \leq c,\ \ \forall  k\geq 1.$$
Since $a_k, b_k$ is positive sequences, we can  conclude
$$a_k \leq c,\ \ \forall \ \ k\geq 1.$$
\end{proof}

\begin{lemma}\label{lem 3}
 Let positive sequence $t_k$ generated in Algorithm 1 via \eqref{fastiterative1} with $t_1 = 1$. Then $t_k\geq (k+1)/2$ for all $k\geq 1$.
 \end{lemma}
\begin{proof}
1) When $k = 1$, $t_1=1\geq \frac{1+1}{2} =1$.\\
2) Suppose that when $k = n$, the conclusion is correct. Then, while $k = n+1$
\begin{equation}
t_{k+1} = \frac{1+\sqrt{1+4t_k^2}}{2} \geq \frac{1+\sqrt{1+4(\frac{k+1}{2})^2}}{2}= \frac{1+\sqrt{1+(k+1)^2}}{2}\geq \frac{(k+1+1)^2}{2}
\end{equation}
According to 1), 2), conclusion is correct.
\end{proof}

\begin{lemma}\label{lem 4}
 For any $\mathbf{Y}\in\mathcal{M}$, the point $\mathbf{X} = p_{\alpha}(\mathbf{Y})$ is a local minimizer of \eqref{oqa} if and only if there exists $\gamma(\mathbf{Y})\in\partial g(\mathbf{X})$, such that \cite{MishraMeyerBachSepulchre2013}
$$
  \mathcal T_{\mathbf{Y}\rightarrow{P_\alpha(\mathbf{Y})}}\big(\mathrm{grad} f(\mathbf{Y})+\alpha L_{\mathbf{Y}}(P_{\alpha}(\mathbf{Y})\big) +\gamma(\mathbf{Y})=\mathbf{0}.
$$
where the vector transport $\mathcal{T}_s$ is isometric means.
\end{lemma}
\begin{proof}
From \eqref{oqa},
\begin{align}
Q_{\alpha}(\mathbf{X},\mathbf{Y}) := &f(\mathbf{Y}) + \big\langle\mathbf{grad}f(\mathbf{Y}), L_{\mathbf{Y}}(\mathbf{X})\big\rangle +\frac{\alpha}{2} \|L_{\mathbf{Y}}(\mathbf{X})\|_{\mathbf{Y}}^2+g(\mathbf{X})
\end{align}
is a convex function.  From \cite{MishraMeyerBachSepulchre2013}, we obtain that the point $\mathbf{X} = P_\alpha(\mathbf{Y})$ is a local minimizer of $Q_{\alpha}(\mathbf{X},\mathbf{Y})$,  if and only if
there exists $\gamma(\mathbf{Y})\in\partial g(\mathbf{X})$,  such that
\begin{equation*}
\mathrm{grad}\big{(}f(\mathbf{Y}) + \langle\mathrm{grad}f(\mathbf{Y}),L_{\mathbf{Y}}(\mathbf{X})\rangle +\frac{\alpha}{2}\|L_{\mathbf{Y}}(\mathbf{X})\|_{\mathbf Y}^2\big{)}\big{|}_{\mathbf{X} = P_\alpha(\mathbf{Y})}+ \gamma(\mathbf{Y})=\mathbf{0}.
\end{equation*}
 Since
 \begin{equation*}
 \begin{split}
 &\mathrm{grad}\big{(}f(\mathbf{Y}) + \langle\mathrm{grad}f(\mathbf{Y}), L_{\mathbf{Y}}(\mathbf{X})\rangle +\frac{\alpha}{2}\|L_{\mathbf{Y}}(\mathbf{X})\|_{\mathbf{Y}}^2\big{)}\big{|}_{\mathbf{X} = p_\alpha(\mathbf{Y})} \\
 =& \mathcal T_{\mathbf{Y}\rightarrow{P_\alpha(\mathbf{Y})}}(\mathrm{grad} f(\mathbf{Y})+\alpha L_{\mathbf{Y}}(P_{\alpha}(\mathbf{Y}))),
 \end{split}
 \end{equation*}
 the conclusion is right.
\end{proof}

\begin{lemma}\label{lem 5}
 Let $\mathbf{Y}_k\in \mathcal{M}$, $k= 1,2,\cdots$ and $\alpha >0$ be such that
$$F(P_\alpha(\mathbf{Y}_k))\leq Q_{\alpha}(P_{\alpha}(\mathbf{Y}_k),\mathbf{Y}_k), $$
Then for any $\mathbf{X}\in \mathcal{M}$
\begin{equation*}
  F(\mathbf{X}) - F(P_{\alpha}(\mathbf{Y}))\geq \frac{\alpha}{2}\|L_{\mathbf{Y}}(P_{\alpha}(\mathbf{Y}))\|_{\mathbf{Y}}^2 -\alpha\langle L_{\mathbf{Y}}(P_{\alpha}(\mathbf{Y})), L_{\mathbf{Y}}(\mathbf{X})\rangle
\end{equation*}
\end{lemma}

\begin{proof}
 From \eqref{searchinequation}, \eqref{convexf},  we have
\begin{equation*}
\begin{split}
  &F(\mathbf{X})-F(P_{\alpha}(\mathbf{Y}))\\
  \geq &F(\mathbf{X})-Q(P_{\alpha}(\mathbf{Y}),\mathbf{Y})\\
  \geq &f(\mathbf{Y}) +\langle \text{grad}f(\mathbf{Y}),L_{\mathbf{Y}}(\mathbf{X}) \rangle+ g(P_\alpha(\mathbf{Y})) +\langle \gamma(\mathbf{Y}),L_{P_\alpha(\mathbf{Y})}(\mathbf{X})\rangle\\
  & -f(\mathbf{Y})-\langle \text{grad}f(\mathbf{Y}), L_{\mathbf{Y}}(P_{\alpha}(\mathbf{Y}))\rangle -\frac{\alpha}{2}\langle L_{\mathbf{Y}}(P_{\alpha}(\mathbf{Y})), L_{\mathbf{Y}}(P_{\alpha}(\mathbf{Y})) \rangle -g(P_\alpha(\mathbf{Y}))\\
  = &\langle \text{grad}f(\mathbf{X}),L_{\mathbf{Y}}(\mathbf{X})-L_{\mathbf{Y}}(P_{\alpha}(\mathbf{Y}))\rangle +\langle L_{P_\alpha(\mathbf{Y})}(\mathbf{Y}),\\
  &\gamma (\mathbf{Y})\rangle-\frac{\alpha}{2}\langle L_{\mathbf{Y}}(P_{\alpha}(\mathbf{Y})), L_{\mathbf{Y}}(P_{\alpha}(\mathbf{Y})) \rangle\\
  =& \langle \text{grad}f(\mathbf{X}), L_{\mathbf{Y}}(\mathbf{X})-L_{\mathbf{Y}}(P_{\alpha}(\mathbf{Y}))\rangle-\langle L_{P_\alpha(\mathbf{Y})}(\mathbf{Y}),T_{\mathbf{Y}\rightarrow{P_\alpha(\mathbf{Y})}}(\text{grad}f(\mathbf{Y})\\
  &+\alpha L_{\mathbf{Y}}(P_{\alpha}(\mathbf{Y})))\rangle -\frac{\alpha}{2}\langle L_{\mathbf{Y}}(P_{\alpha}(\mathbf{Y})),L_{\mathbf{Y}}(P_{\alpha}(\mathbf{Y}))\rangle)\\
  =& \langle \text{grad}f(\mathbf{X}),L_{\mathbf{Y}}(\mathbf{X})-L_{\mathbf{Y}}(P_{\alpha}(\mathbf{Y}))\rangle-\langle L_{\mathbf{Y}}(\mathbf{X})-L_{\mathbf{Y}}(P_{\alpha}(\mathbf{Y}))
  ,\text{grad}f(\mathbf{Y})\\
  &+\alpha L_{\mathbf{Y}}(P_{\alpha}(\mathbf{Y}))\rangle-\frac{\alpha}{2}\langle L_{\mathbf{Y}}(P_{\alpha}(\mathbf{Y})), L_{\mathbf{Y}}(P_{\alpha}(\mathbf{Y})) \rangle)\\
  =&\frac{\alpha}{2}\|L_{\mathbf{Y}}(P_{\alpha}(\mathbf{Y}))\|_{\mathbf{Y}}^2 -\alpha\langle L_{\mathbf{Y}}(P_{\alpha}(\mathbf{Y})), L_{\mathbf{Y}}(\mathbf{X})\rangle
\end{split}
\end{equation*}
\end{proof}
\begin{lemma}\label{lem 6}
The sequences $\{\mathbf{X}_k, \mathbf{Y}_k\}$ generated via Algorithm 1 satisfy for every $k\geq 1$
\begin{equation}
  \frac{2}{\alpha_k}t_k^2v_k -\frac{2}{\alpha_k+1}t_{k+1}^2v_{k+1}\leq \|\mathbf{U}_{k+1}\|_{\mathbf{X}_{k+1}}^2-\|\mathbf{U}_k\|_{\mathbf{X}_k}^2
\end{equation}
where
\begin{equation*}
\begin{split}
& v_k: = F(X_k)-F(X_*)\\
&\mathbf{U}_{k+1}: = \|-(t_{k+1}-1)L_{\mathbf{X}_{k+1}}(\mathbf{X}_{k})- L_{\mathbf{X}_{k+1}}(\mathbf{X}_*)\|_{\mathbf{X}_{k+1}}^2.
\end{split}
\end{equation*}
\end{lemma}
\begin{proof}
  Frist we apply Lemma \eqref{lem 5} at the points $\mathbf{X}:=\mathbf{X}_k, \mathbf{Y}: =\mathbf{Y}_{k+1}$ with $\alpha = \alpha_{k+1}$, and likewise at the points $\mathbf{X}:=\mathbf{X}_*, \mathbf{Y}: =\mathbf{Y}_{k+1}$ to get
\begin{equation*}
\begin{split}
 & 2\alpha_{k+1}^{-1}(v_k-v_{k+1})\\
 \geq &\|L_{\mathbf{Y}}(P_{\alpha}(\mathbf{Y}))\|_{\mathbf{Y}}^2- 2\langle L_{\mathbf{Y}_{k+1}}(\mathbf{X}_{k+1}),L_{\mathbf{Y}_{k+1}}(\mathbf{X}_k)\rangle-2\alpha_{k+1}^{-1}v_{k+1}\\
 \geq & \|L_{\mathbf{Y}}(P_{\alpha}(\mathbf{Y}))\|_{\mathbf{Y}}^2- 2\langle L_{\mathbf{Y}_{k+1}}(\mathbf{X}_{k+1}), L_{\mathbf{Y}_{k+1}}(\mathbf{X}_*)\rangle
\end{split}
\end{equation*}
where we used the fact that $\mathbf{X}_{k+1}=P_{\alpha_{k+1}}(\mathbf{Y}_{k+1})$. Now we multiply the first inequality above by $t_{k+1}-1$ and add it to the second inequality:
\begin{equation*}
\begin{split}
&\frac{2}{\alpha_{k+1}}((t_{k+1}-1)v_k-t_{k+1}v_{k+1})\\
\geq &  t_{k+1}\langle L_{\mathbf{Y}_{k+1}}(\mathbf{X}_{k+1}), L_{\mathbf{Y}_{k+1}}(\mathbf{X}_{k+1})\rangle- 2\langle L_{\mathbf{Y}_{k+1}}(\mathbf{X}_{k+1}), (t_{k+1}-1)L_{\mathbf{Y}_{k+1}}(\mathbf{X}_k)\\
&+L_{\mathbf{Y}_{k+1}}(\mathbf{X}_*)\rangle
\end{split}
\end{equation*}
Multiplying the last inequality by $t_{k+1}$ and using the relation $t_k^2 = t_{k+1}^2-t_{k+1}$, we obtain
\begin{equation*}
\begin{split}
&\frac{2}{\alpha_{k+1}}(t_{k}^2v_k-t_{k+1}^2v_{k+1})\\
\geq &  t_{k+1}^2\langle L_{\mathbf{Y}_{k+1}}(\mathbf{X}_{k+1}), L_{\mathbf{Y}_{k+1}}(\mathbf{X}_{k+1})\rangle - 2\langle t_{k+1}L_{\mathbf{Y}_{k+1}}(\mathbf{X}_{k+1}),(t_{k+1}-1)L_{\mathbf{Y}_{k+1}}(\mathbf{X}_k)\\
&+L_{\mathbf{Y}_{k+1}}(\mathbf{X}_*)\rangle\\
 =& \|t_{k+1}L_{\mathbf{Y}_{k+1}}(\mathbf{X}_{k+1})-(t_{k+1}-1)L_{\mathbf{Y}_{k+1}}(\mathbf{X}_k)-L_{\mathbf{Y}_{k+1}}(\mathbf{X}_*)\|_{\mathbf{Y}_{k+1}}^2\\
 &-\| -(t_{k+1}-1)L_{\mathbf{Y}_{k+1}}(\mathbf{X}_k)-L_{\mathbf{Y}_{k+1}}(\mathbf{X}_*)\|_{\mathbf{Y}_{k+1}}^2
\end{split}
\end{equation*}
Here the vector transport $\mathcal{T}_s$ is isometric means, Therefore,
\begin{equation*}
\begin{split}
 \mathbf{U}_{k+1} &=\|t_{k+1}  L_{\mathbf{Y}_{k+1}}(\mathbf{X}_{k+1})-(t_{k+1}-1)L_{\mathbf{Y}_{k+1}}(\mathbf{X}_k)-L_{\mathbf{Y}_{k+1}}(\mathbf{X}_*)\|_{\mathbf{X}_{k+1}}^2\\
 &=\|-(t_{k+1}-1)L_{\mathbf{X}_{k+1}}(\mathbf{X}_{k})- L_{\mathbf{X}_{k+1}}(\mathbf{X}_*)\|_{\mathbf{X}_{k+1}}^2
 \end{split}
 \end{equation*}
  with $L_{\mathbf{X}_{k}}(\mathbf{Y}_{k+1}) = -((t_k-1)/t_{k+1})L_{\mathbf{X}_k}(\mathbf{X}_{k-1})$,
 \begin{equation*}
 \begin{split}
 \mathbf{U}_{k} &= \| -(t_{k+1}-1)L_{\mathbf{Y}_{k+1}}(\mathbf{X}_k)-L_{\mathbf{Y}_{k+1}}(\mathbf{X}_*)\|_{\mathbf{X}_{k+1}}^2\\
 & = \|(t_{k+1}-1)L_{\mathbf{X}_{k}}(\mathbf{Y}_{k+1})+L_{\mathbf{X}_k}(\mathbf{Y}_{k+1}) -L_{\mathbf{X}_{k}}(\mathbf{X}_*)\|_{\mathbf{X}_k}^2\\
 & = \|-(t_k-1)L_{\mathbf{X}_k}(\mathbf{X}_{k-1})-L_{\mathbf{X}_k}(\mathbf{X}_*)\|_{\mathbf{X}_k}^2
 \end{split}
 \end{equation*}
and with $\alpha_{k+1}\leq \alpha_k$, yields
\begin{equation*}
  \frac{2}{\alpha_{k}}t_{k}^2v_k-\frac{2}{\alpha_{k}}t_{k+1}^2v_{k+1}\geq \mathbf{U}_{k+1}-\mathbf{U}_{k}.
\end{equation*}
\end{proof}
Now we prove the promised improved complexity result of Theorem \ref{thm 1}.

\begin{proof}
Define the quantities
$$a_k: = \frac{2}{\alpha_{k}}t_{k}^2v_k, b_k:=\mathbf{U}_k, c: = \|L_{\mathbf{X}_0}(\mathbf{X}_*)\|_{\mathbf{X}_0}^2 = \|L_{\mathbf{Y}_1}(\mathbf{X}_*)\|_{\mathbf{X}_0}^2 $$
and recall Lemma \eqref{lem 2} that $v_k=F(\mathbf{X}_k)-F(\mathbf{X}_*)$. Then,
$$a_k-a_{k+1}\geq b_{k+1}-b_k$$
and $t_1 = 1$, applying Lemma \eqref{lem 5} to the points $\mathbf{X}:= \mathbf{X}_*, Y:= Y_1$ with $\alpha=\alpha_1$, we get
\begin{equation*}
\begin{split}
 & F(\mathbf{X}_*)-F(\mathbf{p(\mathbf{Y}}_1))\\
 &\geq \frac{\alpha_1}{2}\langle L_{\mathbf{Y}_1}(P_{\alpha}(\mathbf{Y}_1)), L_{\mathbf{Y}_1}(P_{\alpha}(\mathbf{Y}_1))\rangle-\alpha_1 \langle L_{\mathbf{Y}_1}(\mathbf{X}_*),L_{\mathbf{Y}_1}(P_{\alpha}(\mathbf{Y}_1))\rangle
  \end{split}
\end{equation*}
Then
\begin{equation*}
  \begin{split}
    a_1+b_1
    &= \frac{2}{\alpha_1}v_1 + \|L_{\mathbf{X}_1}(\mathbf{X}_*)\|_{\mathbf{X}_1}^2\\
    &= \frac{2}{\alpha_1}(F(\mathbf{X}_1)-F(\mathbf{X}_*))+\|L_{\mathbf{X}_1}(\mathbf{X}_*)\|_{\mathbf{X}_1}^2\\
    &= -\frac{2}{\alpha_1}(F(\mathbf{X}_*)-F(P_\alpha(\mathbf{Y}_1))+\|L_{\mathbf{X}_1}(\mathbf{X}_*)\|_{\mathbf{X}_1}^2\\
    &\overset{(17)}{\leq} -(\|L_{\mathbf{Y}_1}(P_{\alpha}(\mathbf{Y}_1))\|_{\mathbf{Y}_1}^2-2 \langle¡¡L_{\mathbf{Y}_1}(\mathbf{X}_*), L_{\mathbf{Y}_1}(P_{\alpha}(\mathbf{Y}_1)) \rangle) +\|L_{\mathbf{X}_1}(\mathbf{X}_*)\|_{\mathbf{X}_1}^2\\
    &= -\|L_{\mathbf{Y}_1}(\mathbf{X}_1)-L_{\mathbf{Y}_1}(\mathbf{X}_*)\|_{\mathbf{Y}_1}^2+\|L_{\mathbf{Y}_1}(\mathbf{X}_*)\|_{\mathbf{Y}_1}^2+\|L_{\mathbf{X}_1}(\mathbf{X}_*)\|_{\mathbf{X}_1}^2\\
    & = \|L_{\mathbf{Y}_1}(\mathbf{X}_*)\|_{\mathbf{Y}_1}^2\\
    &=c.
  \end{split}
\end{equation*}
 By Lemma \eqref{lem 2} we have for every $k\geq 1$,
\begin{equation*}
  a_k = \frac{2}{\alpha_{k}}t_{k}^2v_k \leq \|L_{\mathbf{X}_0}(\mathbf{X}_*)\|_{\mathbf{X}_0}^2
\end{equation*}
Invoking Lemma \eqref{lem 3}, $ t_k\geq (k+1)/2, \alpha_k\leq \eta L(f) $,  we obtain that
\begin{equation*}
  F(\mathbf{X}_k)-F(\mathbf{X}_*) \leq \frac{2\eta L(f)\|L_{\mathbf{X}_0}(\mathbf{X}_*)\|_{\mathbf{X}_0}^2}{(k+1)^2}
\end{equation*}
\end{proof}
\end{document}